\documentclass[a4paper, 11pt, headings = small, abstract]{scrartcl}
\usepackage[english]{babel}

\usepackage{amscd,amsmath,amsthm,array,hhline,mathrsfs, enumerate, booktabs,fancybox,calc,textcomp,xcolor,graphicx,bbm,xspace,nicefrac,stmaryrd,url,arcs,listings,stmaryrd}
\usepackage[theoremfont]{newpxtext}
\usepackage[vvarbb, upint, bigdelims]{newpxmath}
\usepackage[scr=boondoxupr, scrscaled = 1.05, cal = euler, calscaled =1.05]{mathalpha}
\usepackage[scaled=0.95]{inconsolata}

\DeclareMathAlphabet{\mathsf}{OT1}{\sfdefault}{m}{n}
\SetMathAlphabet{\mathsf}{bold}{OT1}{\sfdefault}{b}{n}
\usepackage[utf8]{inputenc}

\usepackage{enumitem}
\usepackage{etoolbox}
\usepackage[normalem]{ulem}
\usepackage{mathtools}
\usepackage{geometry}
\usepackage{float}

\usepackage{bm}
\usepackage{tikz}
\usepackage[outdir =./]{epstopdf}
\usepackage[citestyle = numeric-comp,bibstyle = numeric, giveninits=true, natbib = true, backend = bibtex,maxbibnames=99,url=false, isbn=false]{biblatex}
\numberwithin{equation}{section}

\usepackage{chngcntr}
\counterwithin{figure}{section}

\usepackage{xcolor}
\usepackage[colorlinks=true, allcolors=myteal]{hyperref}
\definecolor{grey_pers}{RGB}{69 90 100}
\definecolor{WIMgreen}{RGB}{60 134 132}
\definecolor{red_pers}{RGB}{204 37 41}
\definecolor{UMblue}{RGB}{4 47 86}
\definecolor{myteal}{RGB}{0 123 137}
\definecolor{dartmouthgreen}{rgb}{0.05, 0.5, 0.06}\definecolor{cobalt}{rgb}{0.0, 0.28, 0.67}\definecolor{coolblack}{rgb}{0.0, 0.18, 0.39}
\definecolor{glaucous}{rgb}{0.38, 0.51, 0.71}\definecolor{hooker\'sgreen}{rgb}{0.0, 0.44, 0.0}\definecolor{lemonchiffon}{rgb}{1.0, 0.98, 0.8}\definecolor{oucrimsonred}{rgb}{0.6, 0.0, 0.0}\definecolor{radicalred}{rgb}{1.0, 0.21, 0.37}\definecolor{raspberry}{rgb}{0.89, 0.04, 0.36}\definecolor{royalazure}{rgb}{0.0, 0.22, 0.66}
\definecolor{dex}{RGB}{138 18 34}

\theoremstyle{plain}
\newtheorem{theorem}{Theorem}[section]
\newtheorem{proposition}[theorem]{Proposition}
\newtheorem{lemma}[theorem]{Lemma}
\newtheorem{corollary}[theorem]{Corollary}

\theoremstyle{definition}

\theoremstyle{assumption}

\theoremstyle{remark}
\newtheorem{remark}[theorem]{Remark}

\setkomafont{sectioning}{\normalcolor\bfseries}
\setkomafont{descriptionlabel}{\normalcolor\bfseries}
\setkomafont{section}{\Large}
\setkomafont{subsection}{\large}
\setkomafont{subsubsection}{\large}
\setkomafont{paragraph}{\large}
\setkomafont{subparagraph}{\large}
\setkomafont{author}{\large}
\setkomafont{date}{\normalsize}

\def\supp{\operatorname{supp}}

\def\TV{\operatorname{TV}}
\def\H{\mathscr{H}}

\def\E{\mathbb{E}}
\def\G{\mathbb{G}}

\def\N{\mathbb{N}}
\def\M{\mathcal{M}}
\def\P{\mathbb{P}}
\def\N{\mathbb{N}}

\def\R{\mathbb{R}}
\definecolor{darkred}{rgb}{0,0.6,0}
\def\X{\mathbf{X}}

\def\cX{\mathcal{X}}

\def\cS{\mathcal{S}}

\def\Var{\mathrm{Var}}

\newcommand{\fset}{\Lambda}

\newcommand{\cG}{\mathscr{G}}

\newcommand{\ep}{\varepsilon}

\newcommand{\II}{\mathcal{I}}

\newcommand{\PP}{\mathbb{P}}

\newcommand{\Pro}{\mathbb{P}}

\renewcommand{\hat}{\widehat}
\newcommand{\e}{\mathrm{e}}
\renewcommand{\tilde}{\widetilde}%
\renewcommand{\d}{\mathop{}\!\mathrm{d} }
\newcommand{\lebesgue}{\boldsymbol{\lambda}}

\newcommand{\qv}[1]{\langle#1\rangle}
\newcommand{\1}{\mathbf{1}}
\newcommand{\const}{\mathfrak C}

\restylefloat{table}

\newcommand*\diff{\mathop{}\!\mathrm{d} }

\newcommand{\one}{\mathbf{1}}

\newcommand{\vertiii}[1]{{\left\vert\kern-0.25ex\left\vert\kern-0.25ex\left\vert #1
\right\vert\kern-0.25ex\right\vert\kern-0.25ex\right\vert}}

\def\supp{\mathrm{supp}}

\let\originalleft\left
\let\originalright\right
\renewcommand{\left}{\mathopen{}\mathclose\bgroup\originalleft}
\renewcommand{\right}{\aftergroup\egroup\originalright}

\allowdisplaybreaks

\bibliography{./levy_kontrolle} 
\makeatother
\title{\fontsize{16}{19} \selectfont Mixing it up: A general framework for Markovian statistics}
\author{Niklas Dexheimer\thanks{Aarhus University, Department of Mathematics, Ny Munkegade 118, 8000 Aarhus C, Denmark. \newline Email: \href{mailto:dexheimer@math.au.dk}{dexheimer@math.au.dk}/\href{mailto:strauch@math.au.dk}{strauch@math.au.dk}} \and Claudia Strauch\footnotemark[1] \and Lukas Trottner\thanks{Universit\"at Mannheim, Institut f\"ur Mathematik, B6 26, 68159 Mannheim, Germany. \newline Email: \href{mailto:ltrottne@mail.uni-mannheim.de}{ltrottne@mail.uni-mannheim.de}}}
\date{\vspace{-5ex}}

\begin{document}
\maketitle

\begin{abstract}
Up to now, the nonparametric analysis of multidimensional continuous-time Markov processes has focussed strongly on specific model choices, mostly related to symmetry of the semigroup. While this approach allows to study the performance of estimators for the characteristics of the process in the minimax sense, it restricts the applicability of results to a rather constrained set of stochastic processes and in particular hardly allows incorporating jump structures. 
As a consequence, for many models of applied and theoretical interest, no statement can be made about the robustness of typical statistical procedures beyond the beautiful, but limited framework available in the literature. 
To close this gap, we identify $\beta$-mixing of the process and heat kernel bounds on the transition density representing controls on the long- and short-time transitional behavior, resp., as a suitable combination to obtain $\sup$-norm and $L^2$ kernel invariant density estimation rates  matching the well-understood case of reversible multidimenisonal diffusion processes and outperforming density estimation based on discrete i.i.d.\ or weakly dependent data. Moreover, we demonstrate how up to $\log$-terms, optimal $\sup$-norm \textit{adaptive} invariant density estimation can be achieved within our general framework based on tight uniform moment bounds and deviation inequalities for empirical processes associated to additive functionals of  Markov processes.  The underlying assumptions are verifiable with classical tools from stability theory of continuous time Markov processes and PDE techniques, which opens the door to evaluate statistical performance for a vast amount of popular Markov models. We highlight this point by showing how multidimensional jump SDEs with Lévy driven jump part under different coefficient assumptions can be seamlessly integrated into our framework, thus establishing novel adaptive $\sup$-norm estimation rates for this class of processes.
\end{abstract}

\vspace{.2cm}

\normalsize
\section{Introduction}
There exist various probabilistic concepts that permit the investigation of quantitative ergodic properties of Markov processes, providing a number of approaches to analyzing the rate of convergence of the process to equilibrium.
Such results actually present precious tools for an adequate statistical modelling of complex systems. 
Markov models, especially of (jump) diffusion-type, find numerous applications in biology, chemistry, natural resource management, computer vision, Bayesian inference in machine learning, cloud computing and many more \citep{alvarez2004,brosse2018,du2019,durmus2019,fuchs2013,grenander1994,tamura2015,tzen2019}, and ergodicity can usually be seen as some kind of minimum requirement for the development of a fruitful statistical theory. 
While the probabilistic picture of quantitative ergodic properties is now quite clear, there are still open questions regarding the statistical implications. With this paper, we want to contribute to closing this gap, paying particular attention to a general Markovian multidimensional setting.

In contrast to the highly-developed statistical theory for scalar diffusion processes, there are relatively few references for nonparametric or high-dimensional general Markov models.
To not let sampling effects obscure the statistical implications, it is natural to base the statistical analysis in this context on a continuous observation scheme (i.e., one assumes that a complete trajectory of the process is available). 
A substantial point of reference for a thorough statistical analysis of ergodic multivariate diffusion processes is provided by the article \citep{dalrei07} where the fundamental question of asymptotic statistical equivalence is investigated. 
Apart from its principal central statement, the work also nicely demonstrates the implications of probabilistic properties of processes on quantitative statistical results. 
Specifically, heat kernel bounds and the spectral gap inequality are used to prove tight variance bounds for integral functionals which in turn provide fast convergence rates for the specific problem of invariant density estimation.
Similar techniques can be used for the in-depth analysis of other statistical questions such as (adaptive) estimation of the drift vector of an ergodic diffusion (cf.~\citep{str15}, \citep{str16}). 
The results in \citep{dalrei07,str15,str16} are developed for diffusion processes with drift of gradient-type and unit diffusion matrix. 
While in this specific case the reversibility assumption is directly verified, the condition of symmetry of the process presents a significant constraint, in particular for solutions of SDEs with jump noise. 

More recently, a Bayesian approach to drift estimation of multivariate diffusion processes is undertaken in \citep{nickl20} and \citep{giordano2020}. Whilst \citep{giordano2020} work in a reversible setting since their approach relies on placing a Gaussian prior on the potential $B$ of the drift $b = \nabla B$ instead of tackling the drift directly, \citep{nickl20} approach drift estimation for non-reversible diffusions by employing PDE techniques to a penalized likelihood estimator. 
This opens up an excitingly different viewpoint on the statistical handling of multivariate diffusion processes and in case of \citep{nickl20} avoids the need for reversibility, but both approaches  restrict the setting to assumed periodicity of the drift coefficient.
While this assumption (similar to reversibility) can certainly be justified for specific applications, the approach does not yet provide an answer to the question of how to conduct a statistical analysis of multidimensional Markov processes without strong structural constraints on the coefficients. 
From a different perspective, the current preprint \citep{amorino2020invariant} yields the remarkable observation that quantitatively similar statistical results as in the reversible diffusion case can also be proven for jump diffusions with L\'evy-driven jump part, without the need to rely on a reversible or periodic setting, by focusing on assumptions on the characteristics of the process which guarantee exponential ergodicity as the driving force of the statistical approach.

Another branch of the literature that does not consider specific structural assumptions on the process is based on the so called Castellana--Leadbetter condition or variations thereof \citep{casta1986,bosq1997,leblanc97}, which imposes finiteness of the integrated uniform distance between the density of the bivariate law of $(X_0,X_t)$ of a stationary Markov process $\X$ with stationary density $\rho$ and the product density $\rho \otimes \rho$. This assumption yields dimension independent parametric estimation rates of the invariant density and is thus not suitable for our goal to extend the dimension dependent minimax optimal estimation rates for continuous diffusion processes to more general classes of multidimensional Markov processes, introduced below.

Throughout, we suppose that $(\X, (\PP^x)_{x \in \R^d})$ is a non-explosive Borel right Markov process with state space $(\R^d, \mathcal{B}(\R^d))$ and semigroup $(P_t)_{t \geq 0}$ defined by
\[P_t(x,B) \coloneqq \PP^x(X_t \in B), \quad x \in \R^d, B \in \mathcal{B}(\R^d),\]
see Definition 8.1 in \citep{sharpe1988} for an exact characterization. Without going into distracting technical details, it is relevant for our purposes that Borel right Markov processes are possibly the most general class of continuous time, right continuous Markov processes having the strong Markov property that permits deep connections to potential theory. This class of Markov processes includes the more specific class of standard processes, which form the basis of the classical textbook \citep{blumenthal1968}, and even more specifically Feller processes, i.e., Markov processes with a strongly continuous semigroup mapping $\mathcal{C}_0(\R^d)$, the space of continuous functions on $\R^d$ vanishing at infinity, onto itself.
Under regularity assumptions on the coefficients, the exemplary class of (jump) diffusion processes that we study in detail later on belongs to the class of Feller processes and hence falls into our general probabilistic regime. Moreover, due to their natural embedding into potential theory, Borel right Markov processes are the object of stability analysis of time-continuous Markov processes pioneered by Meyn and Tweedie in the 1990s \citep{DownMeynTweedie1995,MeynTweedie1993,MeynTweedie1993b,meyn_tweedie_1993}, in which the long-time behaviour is quantitatively associated with Lyapunov drift criteria. This approach is central to our probabilistic modelling. We therefore work, as a minimal requirement for stability, in an ergodic setting for $\X$ throughout the paper. That is, the following assumption is in place:
\begin{enumerate}[leftmargin=*,label=($\mathscr{A}$0),ref=($\mathscr{A}$0)]
\item
The marginal laws of $\X$ are absolutely continuous, i.e., for any $t > 0$ and $x \in \R^d$, there exists a measurable function $p_t\colon \R^d \times \R^d \to \R_+$ such that
\[P_t(x,B) = \int_B p_t(x,y) \diff{y}, \quad B \in \mathcal{B}(\R^d),\]
and, moreover, $\X$ admits a unique absolutely continuous invariant probability measure $\mu$, i.e., there exists a density $\rho\colon \R^d \to \R_+$ such that $\diff \mu=\rho\d\lebesgue$ and
\[\PP^\mu(X_t \in B) \coloneqq \int_{\R^d} P_t(x,B) \, \mu(\diff x) = \int_{\R^d} \int_B p_t(x,y) \rho(x) \d{y} \d{x} = \int_B \rho(x) \d{x} = \mu(B)\]
for any Borel set $B$.\label{invariant probability measure assumption}
\end{enumerate}
We abbreviate $\PP^\mu = \PP$, $\E^\mu=\E$ and denote $\mu(g) = \int g \diff{\mu}$ for $g \in L^1(\mu)$ or $g \geq 0$. Note that in \ref{invariant probability measure assumption} existence of a density $\rho$ of the invariant distribution $\mu$ is not an additional requirement on $\X$, but is guaranteed by the Radon--Nikodym theorem thanks to the definition of invariance and the existence of densities for the transition operators. 

Turning away from Lyapunov criteria for general ergodic Markov processes, the long-time behaviour of Markovian semigroups is also known to be linked to functional inequalities.
The most familiar setting is the $L^2$ framework with its equivalence between the corresponding Poincar\'e inequalities and exponential decay of the Markovian semigroup.
The relation between both approaches in terms of quantifying ergodic properties of Markov processes is studied in \citep{baketal08}.

We want to understand the interaction between the probabilistic concepts and \textit{statistical} properties. In order to obtain a clear picture and benchmark results that are not distorted by discretization errors, we consider a statistical framework including the standing assumption that a continuous observation of a trajectory $\mathbf{X}^T = (X_t)_{t \in [0,T]}$ of $\X$ is available. For the analysis of statistical methods (e.g., for estimating the characteristics of $\X$), variance bounds and deviation inequalities are of central importance.
Section \ref{sec:basic} focuses on the analysis of the variance of additive functionals of the form $\int_0^tf(X_s)\diff s$ for the ergodic process $\X$.
We introduce sets of general assumptions on transition and invariant density which allow to prove tight variance bounds (cf.~Propositions \ref{prop:varextra} and \ref{prop:varmulti}). Here, we consider an on-diagonal heat kernel bound to regulate the short-time transitional behavior of the process and either local uniform transition density convergence to the invariant distribution at sufficient speed for any dimension $d\in \N$ or exponential $\beta$-mixing in dimension $d \geq 2$ to obtain tight controls on the long-time transitions of the process. The combination of heat kernel bound and local uniform transition density convergence can be interpreted as a localized version of the Castellana--Leadbetter condition that separates the short- and long time effects and considerably weakens the inherent assumptions on the speed at which the law of $X_t$ approaches a singular distribution as $t \downarrow 0$ in higher dimensions. We give a detailed analysis of this condition. We demonstrate how total variation convergence at sufficient speed implies the local uniform transition density assumption and argue that in case of $\mu$-a.s.\ exponential ergodicity of the process, exponential $\beta$-mixing and local uniform transition density convergence are essentially equivalent, giving a homogeneous picture of our different sets of assumptions.   

In Section \ref{sec: deviation} we proceed by showing how the  $\beta$-mixing property of $\X$\textemdash which is satisfied for a wide range of Markov processes appearing in applied and theoretical probability theory\textemdash is reflected in uniform moment bounds on  empirical processes associated to integral functionals of $\X$. 
More precisely, for countable classes $\mathcal{G}$ of bounded measurable functions $g$, we establish an upper bound on
\[\Big(\E\Big[\sup_{g \in \mathcal{G}}\Big\vert\frac{1}{T}\int_0^T g(X_s) \diff{s} - \int g\d\mu\Big\vert^p \Big]\Big)^{1\slash p}, \quad p \geq 1,\]
(cf.~Theorem \ref{Bernstein}) stated in terms of entropy integrals related to $\mathcal{G}$ and the variance of the integral functionals. 
This result holds for $\beta$-mixing Borel right processes on general state spaces without any assumptions on the existence of transition densities, i.e., Assumption \ref{invariant probability measure assumption} is diminished to stationarity which further increases the applicability of our findings for future investigations. Such moment bounds and associated uniform deviation inequalities are generally the focal point for efficient implementation of adaptive estimation procedures, both for the $\sup$-norm as well as the pointwise and integrated $L^2$ risk. In our concrete estimation context, we use the uniform moment bounds together with the variance bounds from Section \ref{sec:basic} to establish oracle-type deviation inequalities for the $\sup$-norm risk of a kernel invariant density estimator that is essential for the adaptive estimation scheme considered in Section \ref{sec: density estimation} that we describe below. Our motivation to study $\sup$-norm estimation is not only rooted in the higher degree of intepretability of such statements compared to the pointwise $L^2$ risk and the general usefulness of mathematical results obtained along the way, but also comes from the observation that certain problems from applied probability can only be handled with statistical tools, when $\sup$-norm estimation bounds of a quantitiy of interest are available. This point is highlighted in the accompanying article \cite{christ21}, where the general framework presented in this paper is implemented for the development of data-driven stochastic optimal control strategies for diffusions and Lévy processes.

Making the mixing behaviour of the process a cornerstone of the statistical analysis is completely natural when comparing to discrete time theory. For discrete observations it is well-established in the field of weak dependence that different sets of mixing assumptions (e.g., $\alpha$-mixing or $\beta$-mixing) and relaxations thereof can produce variance bounds and deviation inequalities that hold up to analogous results from i.i.d.\ observations to yield sharp nonparametric estimation results, see \cite{dedecker2007, rio17} for an overview. Statistically, it is therefore fundamentally interesting whether analysing a continuous time mixing Markov process based on full observations in our framework yields better estimation rates compared to partial observations corresponding to a weakly dependent observation sequence.

Indeed, in presence of the additional analytic tool provided by the heat-kernel bound, we establish in Section \ref{sec: density estimation} that the stationary density of exponentially $\beta$-mixing Markov processes can be estimated in any dimension at optimal rates both wrt.\ $\sup$-norm risk and pointwise $L^2$ risk---where optimality is understood relative to the benchmark minimax rates known for continuous reversible diffusion processes that are faster than the nonparametric rate for well-behaved discretely sampled data. We go even further by showing that in dimension $d \geq 3$---where the optimal bandwidth choice depends on the typically unknown degree of H\"older smoothness $\beta$--- a Lepski type adaptive bandwidth selection scheme proposed in \cite{gini09} for i.i.d.\ data fitted to our needs provides optimal estimation rates up to iterated $\log$-factors (see also \citep{lepski13} for an adaptive scheme for anisotropic $\sup$-norm estimation for i.i.d.\ observations). More precisely, our main result Theorem \ref{theo:invdens} shows that given a kernel estimator $\hat{\rho}_{h,T}$ for the unknown invariant density $\rho$ with bandwidth choice
\[h \equiv h(T) \sim \begin{cases}
    \log^2 T \slash \sqrt{T}, &d = 1,\\
    \log T\slash T^{1\slash 4} , &d = 2,\\
    (\log T \slash T)^{1 \slash (2 \beta + d -2)}, &d \geq 3,\end{cases}
\] 
we have for any $p \geq 1$ and a bounded open domain $D$,
\[\E\Big[\sup_{x \in D} \big \lvert \hat{\rho}_{h,T}(x) - \rho(x) \big\rvert^{p} \Big]^{1/p} \in \begin{cases}
    \mathsf{O}\big(\sqrt{\log T/T}\big), &d=1,\\
    \mathsf{O}\big(\log T/\sqrt T\big), &d=2,\\
    \mathsf{O}\big(\left(\log T/T\right)^{\frac{\beta}{2\beta+d-2}}\big), &d\ge3.
\end{cases}
\]
If for $d \geq 3$ we replace the smoothness-dependent bandwidth choice $h(T)$ by the adaptive selector $\hat{h}_T \equiv \hat{h}{}^{(k)}_T$ introduced in \eqref{est:band0} and the order of the kernel is sufficiently large, then for $\log_{(k)} T$ denoting the $k$-th iterated logarithm,
\[\E\Big[\sup_{x \in D} \big \lvert \hat{\rho}_{\hat{h}_T,T}(x) - \rho(x) \big\rvert \Big] \in \mathsf{O}\bigg(\bigg(\frac{\log_{(k)} T \log T}{T}\bigg)^{\frac{\beta}{2\beta + d -2}} \bigg),\]
where $k \in \N$ can, in principle, be chosen arbitrarily large---which however decreases the size of the set of candidate bandwidths for the adaptive selection procedure given a finite oberservation horizon.
We emphasize that the  logarithmic gap could be avoided if constants appearing in the uniform deviation inequality from Section \ref{sec: deviation} were explicitely calculated. This, however, requires exact knowledge of the ergodic and short time behavior of the process, contradicting a truly adaptive nature of the approach. 

Such $\sup$-norm adaptive multivariate estimation results are completely new and complement adaptive $L^2$ estimation procedures based on model selection considered in \cite{comte02} for discrete time mixing chains and in \cite{amorino2020invariant} for Lévy driven jump-diffusions. We emphasize that \cite{comte02} also consider estimation of continuous time mixing processes in terms of their sampled skeletons. However, our improved adaptive estimation rates in presence of heat kernel bounds demonstrate that such approach can be considerably improved by not taking a Markov chain viewpoint under partial observations but by exploiting continuous time probabilistic structures under full observations.

As a concrete example, we investigate multidimensional SDEs with L\'evy-driven jump part, i.e., Markov processes associated to the solution of 
\begin{equation} \label{eq: levy diff}
\diff{X_t} = b(X_t)\diff{t} + \sigma(X_t) \diff{W_t} + \gamma(X_{t-}) \diff{Z}_t, \quad X_0 = x \in \R^d,
\end{equation}
where $\mathbf{W}$ is $d$-dimensional Brownian motion and $\mathbf{Z}$ is a pure jump L\'evy process independent of $\mathbf{W}$. 
In Section \ref{subsec:ou}, we investigate Lévy driven Ornstein--Uhlenbeck processes as the basic class of Lévy driven jump diffusions with unbounded drift coefficient. In presence of non-trivial Gaussian part and very mild moment assumptions on the Lévy measure, we infer optimal $\sup$-norm and pointwise $L^2$ invariant density estimation results in any dimension. In this case, an adaptive estimation procedure is not necessary, since the invariant density is a smooth function.
In Section \ref{sec: jump diff} we allow for more flexible dispersion and jump coefficients $\sigma,\gamma$ with the price to be paid being boundedness of the drift $b$. By considering solutions $\X$ to \eqref{eq: levy diff} under appropriate assumptions on the coefficients $b,\sigma,\gamma$ and the jump measure associated to $\mathbf{Z}$ we can apply our general statistical results to invariant density estimation for $\X$, thus establishing new results on $\sup$-norm adaptive invariant density estimation for such general jump processes.

In the sequel, we concentrate on guiding the reader through our framework and the accompanied mathematical results. 
All proofs are deferred to the appendices after Section \ref{sec: density estimation}, with more specific references on their exact location at the relevant passages of the main text.

\paragraph{Basic notation.}
A set $B \in \mathcal{B}(\R^d)$ is called $\mu$-full if $\mu(B) = 1$. 
We say that the Borel right Markov process $\X$ is $\mu$-a.s.\ $V$-ergodic at speed $\Xi$ if, for some $\mu$-full set $\fset$, 
\begin{equation}\label{def:Verg}
\lVert P_t(x,\cdot) - \mu \rVert_{\mathrm{TV}} \leq CV(x) \Xi(t), \quad t \geq 0, x \in \fset,
\end{equation}
where $V \colon \R^d \to [0,\infty]$ with $V \one_\Lambda(x) < \infty$ and, for a signed measure $\nu$, $\lVert \nu \rVert_{\mathrm{TV}} \coloneqq \sup_{\lvert f \rvert \leq 1} \lvert \nu(f)\rvert$ denotes its total variation norm. 
If \eqref{def:Verg} holds with $\Xi(t) = (1+t)^{-\alpha}$ for some $\alpha > 0$, we say that $\X$ is $\mu$-a.s.\ $V$-polynomially ergodic of degree $\alpha$. If $\Xi(t)= \mathrm{e}^{-\kappa t}$ for some $\kappa > 0$, then $\X$ is called $\mu$-a.s.\ exponentially ergodic. When $\fset = \R^d$ and $V(x) < \infty$ for any $x \in \R^d$, we just say that $\X$ is ergodic at speed $\Xi$ (resp., polynomially ergodic and exponentially ergodic).

For any multi-index $\alpha\in\N^d$ and $x\in\R^d$, set $|\alpha|=\sum_{i=1}^d\alpha_i$ and $x^\alpha=\prod_{i=1}^dx_i^{\alpha_i}$.
For $\llfloor\beta\rrfloor$ denoting the largest integer \textit{strictly} smaller than $\beta$, introduce the H\"older class on an open domain $D \subset \R^d$
\begin{equation}\label{def:hold}
\mathcal H_D(\beta,\mathsf{L})=\Big\{f\in \mathcal{C}^{\llfloor \beta\rrfloor}(D,\R) : \max_{\lvert \alpha \rvert = \llfloor \beta \rrfloor} \sup_{x,y \in D, x \neq y} \frac{\lvert f^{(\alpha)}(x) - f^{(\alpha)}(y)\rvert}{\lvert x -y \rvert^{\alpha - \llfloor \alpha \rrfloor}} \leq \mathsf{L},\sup_{x \in D} \lvert f(x) \rvert \leq \mathsf{L}\Big\},
\end{equation}
where $f^{(\alpha)} \coloneqq\frac{\partial^{|\alpha|}f}{\partial x_1^{\alpha_1}\ldots\partial x_d^{\alpha_d}}$.
Recall that a kernel function $K\colon\R^d\to\R$ is said to be of order $\ell \in \N$ if, for any $\alpha \in \N^d$ with $\lvert \alpha \rvert \leq \ell$, $x \mapsto x^\alpha K(x)$ is integrable and, moreover,
$\int_{\R^d} K(x) \diff{x} = 1$, $\int_{\R^d} K(x)x^\alpha\d x=0$, for $\alpha \in \N^d, \lvert \alpha \rvert \in \{1,\ldots,\ell\}$.

\section{Basic framework and variance analysis of integral functionals of general Markov processes}\label{sec:basic}
This first section focuses on the analysis of the variance of integral functionals of the form $\int_0^t f(X_s)\diff s$ for the ergodic process $\X=(X_s)_{0\le s\le t}$ under different sets of general assumptions on $\X$ that will carry us through the rest of the paper. Such variance bounds are indispensable tools for statistical applications since (as we will see in Section \ref{sec: deviation}) the variance of integral functionals naturally appears in associated deviation inequalities and related moment bounds and thus requires tight estimates. All proofs for this section can be found in Appendix \ref{app: basic}.

\subsection{Variance analysis under assumptions on transition and invariant density}
Recall the definition of Assumption \ref{invariant probability measure assumption} from the introduction. 
We start by working under the following set of additional assumptions:
\begin{enumerate}[leftmargin=*,label=($\mathscr{A}$\arabic*),ref=($\mathscr{A}$\arabic*)]
\item 
In case $d = 1$, there exists a non-negative, measurable function $\alpha\colon (0,1] \to \R_+$ such that, for any $t \in (0,1]$,
\[
\sup_{x,y \in \R} p_t(x,y) \leq \alpha(t) \quad \text{and} \quad \int_{0+}^1 \alpha(t) \diff{t} = c_1 < \infty,
\]
and, in case $d \geq 2$, there exists $c_2 > 0$ such that the following on-diagonal heat kernel estimate holds true:
\begin{equation} \label{cond1:uni}
\forall t\in (0,1]:\, \sup_{x,y \in \R^d} p_t(x,y) \leq c_2 t^{-d \slash 2}.
\end{equation}
\label{ass: density bound}
\item \label{ass: conv bound} There exists a $\mu$-full set $\fset$ such that for any compact set $\mathcal{S} \subset \R^d$, there exists a non-negative, measurable function $r_{\mathcal{S}} \colon (0,\infty) \to \R_+$ such that
\begin{equation} \label{cond2:uni}
\begin{split}
\forall t > 1:\,  \sup_{x \in \mathcal{S} \cap \fset, y \in \mathcal{S}} \lvert p_t(x,y) - \rho(y) \rvert \leq r_{\mathcal{S}}(t) \hspace{4pt}  \text{with} \hspace{4pt} \int_1^\infty r_{\mathcal{S}}(t) \diff{t} = c_{\mathcal{S}} < \infty.
\end{split}
\end{equation}
\end{enumerate}
An essential aspect of the statistical analysis of stochastic processes is the influence of the dimension of the underlying process.
It is known that certain phenomena (as compared, e.g., to estimation based on i.i.d.\ observations) occur in the one-dimensional case.
However, these phenomena can usually only be detected by means of specific techniques that take advantage of the unique probabilistic characteristics of scalar processes such as local time for one-dimensional diffusion processes.
A ``standardized'' statistical framework which covers all dimensions with similar conditions cannot capture these phenomena.
Our assumptions may therefore be understood as an attempt to find general conditions that make no reference to dimension or process specific phenomena, yet yield variance bounds which are tight enough to allow proving optimal convergence rates for nonparametric procedures.

In this regard, they should be compared to the Castellana--Leadbetter condition \citep{casta1986} requiring that
\begin{equation}\label{eq: casta}
\int_{(0,\infty)} \sup_{x,y \in \R^d} \lvert \rho(x) p_t(x,y) - \rho(x)\rho(y)\rvert \diff{t} < \infty,
\end{equation}
and which allows $L^2$ estimation of the invariant density via a kernel estimator at parametric (or \textit{superoptimal} \citep{bosqbook}) rate $1/T$ in \textit{any} dimension $d \geq 1$. Since \ref{ass: density bound} implies that $\rho$ is bounded, \ref{ass: conv bound} can be understood as a  localized, unweighted alternative to \eqref{eq: casta} away from $0$, which captures the mixing behaviour of the process as we discuss below. Our assumption \ref{ass: density bound} corresponds to the integral part of \eqref{eq: casta} close to $0$ and guarantees that the distribution of $X_t$ is not too close to a singular distribution. However, in dimension $d \geq 2$ this assumption is much milder than \eqref{eq: casta} since heat kernel bounds on the transition density are quite common for many multdimensional Markov processes such as strong solutions of (jump) SDEs. On the other hand, \eqref{eq: casta} is too strong for such Markov processes, since, e.g., the minimax optimal $L^2$ rate for multivariate diffusions processes is known to be worse than $1/T$ and hence the variance bound implied by \eqref{eq: casta} cannot be achieved.

Also note that the transition density bounds formulated in \ref{ass: density bound} are  weak compared to related literature dealing with statistical estimation of jump processes. 
E.g., \citep{amorino2020invariant} construct their assumptions on the coefficients and the jump measure of a $d$-dimensional L\'evy-driven jump diffusion to guarantee a heat kernel-type estimate of the form
\[
p_t(x,y) \lesssim t^{-d\slash 2} \mathrm{e}^{-\lambda \frac{\lVert y - x \rVert^2}{t}} + \frac{t}{\vert \sqrt{t} + \lVert y -x \rVert\vert^{d+ \alpha}}, \quad x,y \in \R^d, t\in (0,T],
\]
for the estimation horizon $T> 0$, where $\alpha \in (0,2)$ is the self-similarity index of a strictly $\alpha$-stable L\'evy process whose L\'evy measure is assumed to dominate the L\'evy measure governing the jumps of the SDE. 
Clearly, this condition is stronger than what we require and is fitted to the concrete probabilistic setting.
The reason for this specific choice becomes apparent from Corollary \ref{coroll: heat jump} in Section \ref{sec: jump diff}, but our approach reveals that \ref{ass: density bound} is sufficient to obtain tight variance bounds in a general multivariate setting.
Let us now give the variance bounds implied in our framework.

\begin{proposition}\label{prop:varextra}
Suppose that \ref{ass: density bound} and \ref{ass: conv bound} are satisfied, and let $f$ be a bounded function with compact support $\mathcal S$ fulfilling $\lebesgue(\mathcal S)<1$.
Then, there exists a constant $C>0$, such that, for any $T>0$,
\begin{equation}\label{varpsi}
\Var\left(\int_0^Tf(X_t)\diff t\right) \leq C(1 \vee c_{\mathcal{S}})T\|f\|_\infty^2\lebesgue(\mathcal S)\mu(\mathcal S)\psi_d^2(\lebesgue(\mathcal S)), \hspace{3pt}\text{with }\psi_d(x)\coloneqq\begin{cases}
1, &d=1,\\
\sqrt{1+\log(1/x)}, & d=2,\\
x^{\frac1d-\frac12}, &d\ge 3,\end{cases}
\end{equation}
where the variance is taken with respect to $\PP$.
\end{proposition}

To get an impression of the usefulness of the above result, let us discuss the relation of the local uniform transition density convergence assumption \ref{ass: conv bound} to more general and often conveniently verifiable stability conditions on $\X$.
In \cite{vere99}, conditions on the characteristic function $\varphi_{X_t}^x(\lambda) \coloneqq \E^x[\exp(\mathrm{i}\langle X_t, \lambda)]$ of $X_t$ and the Fourier transform $\{\mathscr{F}\mu\}(\lambda) = \int_{\R^d} \mathrm{e}^{\mathrm{i}\langle x,\lambda \rangle} \, \mu(\diff{x})$ were formulated in the scalar setting $d=1$ that imply finiteness of the integral part away from $0$ in the Castellana--Leadbetter condition \eqref{eq: casta}. A straightforward adaption to our multivariate localized setting yields the following result, with the proof being omitted. 

\begin{lemma} \label{lem: vere}
Suppose that $\X$ is $V$-polynomially ergodic of degree $\gamma_1 > q/(q-1)$ for some locally bounded function $V$ and $q > 1$. If there exists $\gamma_2 > qd$ and a locally bounded function $V$ such that
\begin{enumerate}[label = ($\mathscr{V}$\arabic*), ref = ($\mathscr{V}$\arabic*)]
\item \label{ass: char f0} $\lvert \varphi_{X_t}^x(\lambda) - \{\mathscr{F} \mu\}(\lambda)\rvert \leq V(x) (1+ t)^{-\gamma_1}, \quad t \geq 1, x, \lambda \in \R^d$
\item \label{ass: char f}  
$\lvert \varphi^x_{X_t}(\lambda) \rvert \vee \lvert \{\mathscr{F} \mu \}(\lambda) \rvert \lesssim (1+ \lVert \lambda \rVert)^{-\gamma_2}, \quad x,\lambda \in \R^d, t \geq 1,$
\end{enumerate}
then \ref{ass: conv bound} is satisfied with $\fset = \R^d$, $r_{\mathcal{S}}(t) \sim \sup_{x \in \mathcal{S}} V(x) (1+t)^{-\gamma_1}$ for compacts $\mathcal{S}$. 
\end{lemma}
Note that \ref{ass: char f} implies that the Fourier transforms of $P_t(x,\cdot)$ and $\mu$ are integrable and hence the Fourier inversion theorem guarantees that  continuous bounded transition and invariant densities exist. Moreover, as remarked in \cite{vere99}, \ref{ass: char f0} is fulfilled whenever $\X$ is $V$-polynomially ergodic with rate $\gamma_1 > 1$. 

Condition \ref{ass: char f} is quite natural in a statistical estimation context since it essentially encodes a certain amount of smoothness of the transition and stationary density. However, the following simple observation demonstrates that the additional growth conditions on the characteristic function are not needed in presence of sufficiently fast total variation convergence.
Concerning the specific set of assumptions \ref{invariant probability measure assumption}--\ref{ass: conv bound}, it is established with this result in Section \ref{subsec:ou} that they are satisfied, e.g., for a large class of multivariate L\'evy-driven Ornstein--Uhlenbeck processess. 

\begin{lemma} \label{prop: poisson}
Suppose that $\lVert p_1 \rVert_\infty < \infty$ and that $\X$ is $\mu$-a.s.\ $V$-ergodic at speed $\Xi$ such that $V\one_\fset$ is locally bounded  and $\int_0^\infty \Xi(t) < \infty$. Then, \ref{ass: conv bound} holds with 
\[r_{\mathcal{S}}(t) = 2C\lVert p_1 \rVert_\infty \sup_{x \in \mathcal{S} \cap \fset} V(x) \Xi(t-1), \quad t > 1.\]
\end{lemma}

Recall that the stationary Markov process $\X$ is said to be $\beta$-mixing if
\[\beta(t) \coloneqq \int_{\R^d} \Vert P_t(x,\cdot)-\mu(\cdot)\Vert_{\TV}\, \mu(\diff x) \underset{t \to \infty}{\longrightarrow} 0.\]
If there exist constants $\kappa, c_\kappa > 0$ such that $\beta(t) \leq c_\kappa \mathrm{e}^{-\kappa t}$ for any $t > 0$, then $\X$ is said to be exponentially $\beta$-mixing, which is always the case for $\mu$-a.s.\ $V$-exponentially ergodic Markov processes provided $\mu(V) < \infty$. Here, $\mu(V) < \infty$ is not a restriction since $V$ and $\kappa > 0$ can always be chosen such that $\mu(V) < \infty$, which follows from a straightforward extension of \cite[Theorem 6.14.(iii)]{nummelin1984} to the continuous time case. By the same theorem, the converse is also true, i.e., if $\X$ is exponentially $\beta$-mixing, then $\X$ is $\mu$-a.s.\ $V$-exponentially ergodic. See also \cite[Lemma 8.9]{chen05} for these statements. Exponential $\beta$-mixing is formulated as assumption \ref{ass: beta mix} in the next section and will be one of the pillars of our statistical analysis for the $\sup$-norm risk. It is therefore critical for us to understand the exact relationship between exponential $\beta$-mixing and \ref{ass: conv bound}. To this end, as a partial converse to Lemma \ref{prop: poisson}, we explore in Appendix \ref{sec:betamix} under which additional (quite natural) conditions, \ref{ass: conv bound} implies the exponential $\beta$-mixing property of $\X$.
Our main findings, taking account of Lemma \ref{prop: poisson}, Appendix \ref{sec:betamix} and the developments in Section \ref{subsec: beta}, are summarized in Figure \ref{fig: variance bounds}.

\tikzset{every picture/.style={line width=0.75pt}} 

\begin{figure}[ht]
\centering


\begin{tikzpicture}[x=0.75pt,y=0.75pt,yscale=-1,xscale=1]

\draw   (160.2,138) .. controls (160.2,133.58) and (163.78,130) .. (168.2,130) -- (230.7,130) .. controls (235.12,130) and (238.7,133.58) .. (238.7,138) -- (238.7,162) .. controls (238.7,166.42) and (235.12,170) .. (230.7,170) -- (168.2,170) .. controls (163.78,170) and (160.2,166.42) .. (160.2,162) -- cycle ;
\draw   (282.2,194) .. controls (282.2,189.58) and (285.78,186) .. (290.2,186) -- (352.7,186) .. controls (357.12,186) and (360.7,189.58) .. (360.7,194) -- (360.7,218) .. controls (360.7,222.42) and (357.12,226) .. (352.7,226) -- (290.2,226) .. controls (285.78,226) and (282.2,222.42) .. (282.2,218) -- cycle ;
\draw   (160.2,249) .. controls (160.2,244.58) and (163.78,241) .. (168.2,241) -- (230.7,241) .. controls (235.12,241) and (238.7,244.58) .. (238.7,249) -- (238.7,273) .. controls (238.7,277.42) and (235.12,281) .. (230.7,281) -- (168.2,281) .. controls (163.78,281) and (160.2,277.42) .. (160.2,273) -- cycle ;
\draw   (399,138) .. controls (399,133.58) and (402.58,130) .. (407,130) -- (469.5,130) .. controls (473.92,130) and (477.5,133.58) .. (477.5,138) -- (477.5,162) .. controls (477.5,166.42) and (473.92,170) .. (469.5,170) -- (407,170) .. controls (402.58,170) and (399,166.42) .. (399,162) -- cycle ;
\draw   (401.4,246.8) .. controls (401.4,242.38) and (404.98,238.8) .. (409.4,238.8) -- (471.9,238.8) .. controls (476.32,238.8) and (479.9,242.38) .. (479.9,246.8) -- (479.9,270.8) .. controls (479.9,275.22) and (476.32,278.8) .. (471.9,278.8) -- (409.4,278.8) .. controls (404.98,278.8) and (401.4,275.22) .. (401.4,270.8) -- cycle ;
\draw    (320.73,228.78) .. controls (317.36,233.92) and (314.24,238.18) .. (310.78,241.74) .. controls (299.28,253.61) and (284.15,257.86) .. (244.27,261.95) ;
\draw [shift={(241.8,262.2)}, rotate = 354.28999999999996] [fill={rgb, 255:red, 0; green, 0; blue, 0 }  ][line width=0.08]  [draw opacity=0] (8.04,-3.86) -- (0,0) -- (8.04,3.86) -- (5.34,0) -- cycle    ;
\draw [shift={(322.4,226.2)}, rotate = 122.5] [fill={rgb, 255:red, 0; green, 0; blue, 0 }  ][line width=0.08]  [draw opacity=0] (8.04,-3.86) -- (0,0) -- (8.04,3.86) -- (5.34,0) -- cycle    ;
\draw    (323.4,186.2) .. controls (308.43,170.44) and (298.11,146.92) .. (244.29,151.57) ;
\draw [shift={(241.8,151.8)}, rotate = 354.28999999999996] [fill={rgb, 255:red, 0; green, 0; blue, 0 }  ][line width=0.08]  [draw opacity=0] (8.04,-3.86) -- (0,0) -- (8.04,3.86) -- (5.34,0) -- cycle    ;
\draw    (200.2,129.4) .. controls (239.6,101.82) and (394.65,100.63) .. (437.52,127.36) ;
\draw [shift={(439.4,128.6)}, rotate = 214.99] [fill={rgb, 255:red, 0; green, 0; blue, 0 }  ][line width=0.08]  [draw opacity=0] (8.04,-3.86) -- (0,0) -- (8.04,3.86) -- (5.34,0) -- cycle    ;
\draw    (201,280.6) .. controls (237.06,310.39) and (394.52,307.91) .. (437.7,281.45) ;
\draw [shift={(440.2,279.8)}, rotate = 504.46] [fill={rgb, 255:red, 0; green, 0; blue, 0 }  ][line width=0.08]  [draw opacity=0] (8.04,-3.86) -- (0,0) -- (8.04,3.86) -- (5.34,0) -- cycle    ;
\draw    (478.4,151) .. controls (495.76,174.4) and (503.79,227.46) .. (481.94,259.02) ;
\draw [shift={(480.2,261.4)}, rotate = 307.57] [fill={rgb, 255:red, 0; green, 0; blue, 0 }  ][line width=0.08]  [draw opacity=0] (8.04,-3.86) -- (0,0) -- (8.04,3.86) -- (5.34,0) -- cycle    ;
\draw    (160.8,150.2) .. controls (139.15,176.72) and (138.61,229.86) .. (157.88,258.45) ;
\draw [shift={(159.4,260.6)}, rotate = 233.39] [fill={rgb, 255:red, 0; green, 0; blue, 0 }  ][line width=0.08]  [draw opacity=0] (8.04,-3.86) -- (0,0) -- (8.04,3.86) -- (5.34,0) -- cycle    ;

\draw (164.2,134.5) node [anchor=north west][inner sep=0.75pt]   [align=left] {\begin{minipage}[lt]{50.5pt}\setlength\topsep{0pt}
\begin{center}
{\scriptsize $\scriptsize{\X}$ fulf.\ \ref{ass: conv bound} w.}\\{\scalebox{.75}{$r_{{\mathcal{S}}}(t) = C_{\cS}\mathrm{e}^{-\kappa_{\cS}t}$}}
\end{center}

\end{minipage}};
\draw (284.2,192) node [anchor=north west][inner sep=0.75pt]   [align=left] {\begin{minipage}[lt]{53.59pt}\setlength\topsep{0pt}
\begin{center}
{\scriptsize $\scriptsize{\X}$ is $\mu$-a.s.\ }\\{\scriptsize $V$-exp.\ ergodic}
\end{center}

\end{minipage}};
\draw (167,245.2) node [anchor=north west][inner sep=0.75pt]   [align=left] {\begin{minipage}[lt]{47.37pt}\setlength\topsep{0pt}
\begin{center}
{\scriptsize $\scriptsize{\X}$ is exp. }\\{\scriptsize $\scriptsize{\beta}$-mixing}
\end{center}

\end{minipage}};
\draw (393,134.6) node [anchor=north west][inner sep=0.75pt]   [align=left] {\begin{minipage}[lt]{67.37pt}\setlength\topsep{0pt}
\begin{center}
{\scriptsize Variance bound}\\{\scriptsize \eqref{varpsi} holds}
\end{center}

\end{minipage}};
\draw (393,243.6) node [anchor=north west][inner sep=0.75pt]   [align=left] {\begin{minipage}[lt]{67.37pt}\setlength\topsep{0pt}
\begin{center}
{\scriptsize Variance bound}\\{\scriptsize \eqref{eq: var mix} holds}
\end{center}

\end{minipage}};

\draw (259.26,107.05) node [anchor=north west][inner sep=0.75pt]  [rotate=-27.05] [align=left] {\begin{minipage}{70.09pt}\setlength\topsep{0pt}
\begin{center}
{\scriptsize $\lVert p_1 \rVert_\infty < \infty,$}\\{\scriptsize $V$ loc.\ bound.}
\end{center}

\end{minipage}};
\draw (308.6,85.8) node [anchor=north west][inner sep=0.75pt]   [align=left] {\begin{minipage}[lt]{10.37pt}\setlength\topsep{0pt}
\begin{center}
{\scriptsize \ref{ass: density bound}}
\end{center}

\end{minipage}};
\draw (290.6,303) node [anchor=north west][inner sep=0.75pt]   [align=left] {\begin{minipage}[lt]{60.37pt}\setlength\topsep{0pt}
\begin{center}
{\scriptsize \ref{ass: density bound},$d \geq 2$}
\end{center}

\end{minipage}};
\draw (498.8,190.8) node [anchor=north west][inner sep=0.75pt]   [align=left] {\begin{minipage}[lt]{10.37pt}\setlength\topsep{0pt}
\begin{center}
{\scriptsize $\scriptsize{\lVert \rho \rVert_\infty < \infty}$}
\end{center}

\end{minipage}};
\draw (88.4,175.6) node [anchor=north west][inner sep=0.75pt]   [align=left] {\begin{minipage}[lt]{41.01pt}\setlength\topsep{0pt}
\begin{center}
{\scriptsize $\scriptsize{\X}$ is an ergodic }\\{\scriptsize $\scriptsize{T}$-process, $\fset = \R^d$}
\end{center}

\end{minipage}};


\end{tikzpicture}
\caption{Overview of interplay between variance bound results, assumptions and stability concepts} \label{fig: variance bounds}
\end{figure}
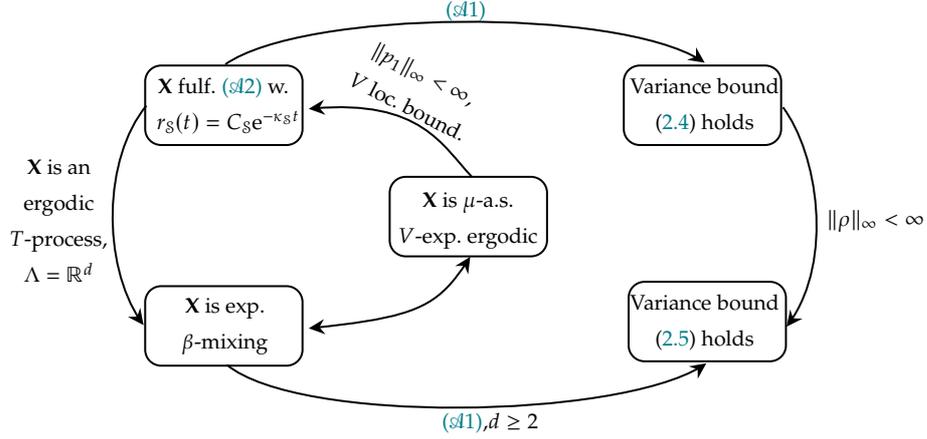
A clear picture is drawn, demonstrating that local uniform transition density convergence at exponential speed is intimately connected with exponential $\beta$-mixing of the process---both concepts having $\mu$-a.s.\ exponential ergodicity as the driving force behind them in most concrete applications. Both conditions \ref{ass: conv bound} and \ref{ass: beta mix} gain substantial additional statistical power via the smoothing assumption \ref{ass: density bound}, which allows obtaining tight variance bounds that yield superior estimation properties under continuous observations compared to incomplete information via sampling procedures, as will be demonstrated in Section \ref{sec: density estimation}. Moreover, the slightly more specific localized Castellana--Leadbetter condition provides the advantage of optimal estimation also in the scalar case $d=1$ and wrt the $L^2$ risk under less restrictive assumptions on the speed of convergence of the process (polynomial is sufficient) in any dimension, which justifies us studying this concept separately from exponential $\beta$-mixing.

\subsection{Variance analysis under exponential $\beta$-mixing} \label{subsec: beta}
In this subsection, we specify our study to multidimensional stochastic processes by restricting the analysis to dimension $d\ge2$. While we further assume that the on-diagonal heat kernel bound on the transition density \eqref{cond1:uni} from \ref{ass: density bound} still holds, we drop the transition density rate assumption \ref{ass: conv bound} and instead impose exponential $\beta$-mixing of $\X$.
Note that this is implied by \ref{ass: conv bound} under suitable technical conditions on $\X$ (see Figure \ref{fig: variance bounds} and Propositions \ref{prop:betamix} and \ref{prop2:betamix} in Appendix \ref{sec:betamix}).
\begin{enumerate}[leftmargin=*,label=($\mathscr{A}\beta$),ref=($\mathscr{A}\beta$)]
\item \label{ass: beta mix}
The process $\X$ started in the invariant measure $\mu$ is exponentially $\beta$-mixing, i.e., there exist constants $c_\kappa,\kappa>0$ such that
\[\int \Vert P_t(x,\cdot)-\mu \Vert_{\TV}\,\mu(\diff x)\ \le\ c_\kappa\e^{-\kappa t}, \quad t \geq 0.\]
\end{enumerate}
Let us emphasize that in presence of the heat kernel bound \ref{ass: density bound}, Lemma \ref{lemma: invariant bounded} below shows that Assumption \ref{invariant probability measure assumption} is strengthened to the existence of a \textit{bounded} invariant density since the transition density of any skeleton chain is uniformly bounded for fixed $t > 0$. That is, the following assumption is in place.
\begin{enumerate}[leftmargin=*,label=($\mathscr{A}0+$),ref=($\mathscr{A}0+$)]
\item Assumption \ref{invariant probability measure assumption} holds and the invariant density has a bounded version $\rho$, i.e., $\lVert\rho\rVert_\infty < \infty.$\label{ass: bounded invariant}
\end{enumerate}

\begin{lemma}\label{lemma: invariant bounded}
Assume that $\X$ has an invariant distribution $\mu$ and that there is some $\Delta>0$ such that the transition density $p_\Delta$ exists and $\sup_{x,y\in \R^d}p_\Delta(x,y)\leq c$ for some constant $c>0$. Then, $\mu$ admits a bounded density.
\end{lemma}

The next result gives a tight variance bound on the integral $\int_0^T f(X_t)\diff{t}$ under $\beta$-mixing.
Its effectiveness for $\sup$-norm estimation of general Markov processes will be demonstrated in Section \ref{sec: density estimation}. 
Note in particular that, using boundedness of $\rho$ under \ref{invariant probability measure assumption} and \ref{ass: density bound}, the same rate can be obtained under \ref{ass: conv bound} from Proposition \ref{prop:varextra}.
Recall the definition of $\psi_d\colon(0,\mathrm{e})\to\R_+$ in \eqref{varpsi}.

\begin{proposition}\label{prop:varmulti}
Grant assumptions \ref{ass: density bound} and \ref{ass: beta mix}, and let $f$ be a bounded function with compact support $\mathcal S$ fulfilling $\lebesgue(\mathcal S)<1$.
Then, for any $d\ge2$, there exists a constant $C>0$ not depending on $f$ such that, for any $T > 0$,
\begin{equation} \label{eq: var mix}
\operatorname{Var}\left(\int_0^T f(X_t)\d t \right)\leq C T \Vert f\Vert_\infty^2\Vert\rho\Vert_\infty\lebesgue^2(\cS)\psi_d^2(\lebesgue(\cS)).
\end{equation}
\end{proposition}

\paragraph{Notation.}
Throughout the sequel, we denote by $\bm{\Sigma}$ the class of non-explosive, exponentially $\beta$-mixing Borel right Markov processes $\X$ such that assumptions \ref{invariant probability measure assumption} and \ref{ass: density bound} hold (and hence \ref{ass: bounded invariant} is in place, i.e., the invariant density $\rho$ is bounded).
Moreover, in dimension $d=1$ we assume that \ref{ass: conv bound} is in place with a rate function $r_{\mathcal{S}}$ which is monotone wrt the compact sets $\mathcal{S}$ in the sense that 
\begin{equation}\label{eq: mon supp}
\mathcal{S}_1 \subset \mathcal{S}_2 \implies c_{\mathcal{S}_1} = \int_1^\infty r_{\mathcal{S}_1}(t) \diff{t} \leq \int_1^\infty r_{\mathcal{S}_2}(t) \diff{t} = c_{\mathcal{S}_2} < \infty.
\end{equation}
Alternatively, if we do not want to restrict to exponentially $\beta$-mixing processes, consider the class of processes $\bm{\Theta}$ consisting of $d$-dimensional non-explosive Borel right processes such that \ref{invariant probability measure assumption}--\ref{ass: conv bound} hold, where again the constants $c_{\mathcal{S}}$ appearing in \ref{ass: conv bound} satisfy \eqref{eq: mon supp}. Note that if $\tilde{\bm{\Theta}}$ is the restriction of $\bm{\Theta}$ containing the class of processes $\X$ satisfying the assumptions of Proposition \ref{prop:betamix} or Proposition \ref{prop2:betamix}, then  $\tilde{\bm{\Theta}} \subset \bm{\Sigma}$.

\section{Uniform moment bounds}\label{sec: deviation}
We now turn to deriving uniform moment bounds for integral functionals of the ergodic process $\X$.
These are intimately connected with Bernstein-type tail inequalities, which due to their crucial importance for many probabilistic and statistical applications---such as the derivation of limit theorems or upper bound statements for nonparametric estimation procedures---have been excessively studied in the literature (see Section 1.1 of \citep{gaoetal13} for an overview).
Both a Lyapunov function method and a functional inequalities approach can be used for deriving results on the concentration behaviour of additive functionals of $\X$.
\citep{catgui08} establish non-asymptotic deviation bounds for
\[\PP\left(\Big|\int_0^t f(X_s)\diff s-\int f\d\mu\Big| \ge r\right),\qquad f\in L^1(\mu),\]
using different moment assumptions for $f$ and regularity conditions for $\mu$, ``regularity'' referring to the condition that $\mu$ may satisfy various functional inequalities (F-Sobolev, generalized Poincar\'e, etc.).
In a symmetric Markovian setting and assuming a spectral gap, Lezaud \citep{lez01} uses Kato's perturbation theory for proving Bernstein-type concentration inequalities for empirical means of the form $\int_0^t f(X_s)\diff s$, the upper bound depending on the asymptotic variance of $f$.
Amongst other methods, \citep{gaoetal13} exploit both a Lyapunov function method and a functional inequalities approach for extending Lezaud's result to inequalities for possibly unbounded $f$.
Going beyond the symmetric case, Lyapunov-type conditions can also be used for verifying exponential mixing properties, paving the way to generalizing concentration results based on independent observations to the dependent case. For corresponding results for discrete random  (Markov) sequences under different mixing or ergodicity assumptions, we refer to \cite{adam08,adam15,bertail19, clemencon01, dedecker2015, merlevede09,lema2020, samson00}

\subsection{General framework}
Our main focus in this subsection is on deriving corresponding uniform moment inequalities of empirical processes, using merely the previously introduced assumptions (in particular, the $\beta$-mixing property), and without imposing any additional conditions on the process. We emphasize that for this section no assumption on the existence of transition or invariant densities is needed, but that we only work within an ergodic $\beta$-mixing framework. Moreover, the results are established for $\beta$-mixing Markov processes with arbitrary topological state space $\mathcal{X}$, not necessarily equal to $\R^d$, and general mixing rate. That is, we suppose in this section that 
\[\beta(t) = \int_{\cX} \lVert P_t(x,\cdot) - \mu \rVert_{\mathrm{TV}}\, \mu(\diff{x}) \leq \Xi(t),\]
for some rate function $\Xi(t)$ decreasing to $0$ as $t \to \infty$.
We aim to prove moment bounds for suprema of the form
\[
\sup_{g\in\mathcal G}|\G_t(g)|\eqqcolon \|\G_t\|_{\mathcal G},\quad\text{ for } \G_t(g)\coloneqq \frac{1}{\sqrt{t}}\int_0^tg(X_s)\d s,
\]
where the supremum is taken over entire (possibly infinite-dimensional) function classes $\mathcal G\subset \mathcal{B}_b(\cX)$ of $\mu$-centered measurable bounded functions on $\cX$.
Similarly to \citep{bar10} and \citep{dirk15}, we apply the generic chaining device for the derivation of our result. The basic strategy of the proof is splitting the integral into blocks of length $m_t$, construct an independent Berbee coupling based on the $\beta$-mixing property as described in Viennet \citep{viennet}, and then use the classical Bernstein inequality for i.i.d.\ random variables for the coupled integral blocks to drive the chaining procedure from \citep{dirk15}. The use of Berbee's coupling lemma is a well-established method for studying empirical processes of discrete $\beta$-mixing sequences, see \cite[Chapter 8]{rio00}, and has recently been employed in \cite{amorino2020invariant} for establishing $L^2$ oracle bounds for an adaptive estimator of the invariant density of a class of exponentially $\beta$-mixing L\'evy-driven jump diffusions.

We now formulate a crucial tool for deriving upper bounds on the \emph{$\sup$-norm risk} of estimators of the invariant density of processes $\X\in\bm{\Sigma}$.
Our final moment bound on the supremum of the process $\G_t$ is stated in terms of entropy integrals of the indexing function class $\mathcal G$. In many applications, the corresponding assumption is straightforward to verify.
For any given $\ep>0$, denote by $\mathcal N(\ep,\mathcal G,d)$ the covering number of $\mathcal G$, i.e., the smallest number of balls of $d$-radius $\ep$ needed to cover $\mathcal G$.
Furthermore, given $f,g\in\mathcal G$, let $d_\infty(f,g)\coloneqq \|f-g\|_\infty$ and
\[d_{\G,t}^2(f,g)\coloneqq\sigma_t^2(f-g),\text{ where } \sigma_t^2(f)\coloneqq\Var\left(\frac{1}{\sqrt t}\int_0^{t}f(X_s)\d s\right).\]

\begin{theorem}\label{Bernstein}
Suppose that $\X$ is $\beta$-mixing with rate function $\Xi(t)$. Let $\mathcal G$ be a countable class of bounded real-valued functions with $\mu(g)=0$ and let $m_t \in (0, t\slash 4]$.
Then, there exist $\tau \in [m_t, 2m_t]$ and constants $\tilde C_1,\tilde C_2>0$ such that, for any $1\le p<\infty$,
\begin{equation} \label{eq: unimom}
\begin{split}
\left(\E\left[\|\G_t\|_{\mathcal G}^p\right]\right)^{1/p}
&\leq \tilde C_1\int_0^\infty\log\mathcal N\big(u,\mathcal G,\tfrac{2m_t}{\sqrt{t}}d_\infty\big)\d u+\tilde C_2\int_0^\infty\sqrt{\log \mathcal N(u,\mathcal G,d_{\G,\tau})}\diff u\\
&\qquad+4\sup_{g\in\mathcal G}\Big(\frac{2m_t}{\sqrt{t}}\|g\|_\infty \tilde c_1p+ \lVert g\rVert_{\mathbb{G},\tau}\tilde c_2\sqrt p +\frac{1}{2}\lVert g \rVert_{\infty} \sqrt{t} \Xi(m_t)^{1/p}\Big),
\end{split}
\end{equation}
for positive constants $\tilde c_1,\tilde c_2$ defined in \eqref{def:c1c2}. 
\end{theorem}

\begin{remark}
Consider $p =1$ and the specific choice of $m_t = \kappa^{-1} \log t$ in case of exponential $\beta$-mixing rate $\Xi(t) = c_\kappa \exp(-\kappa t)$. 
Then, the above result implies that
\[
\E\left[\|\G_t\|_{\mathcal G}\right] \lesssim \int_0^\infty\log\mathcal N\big(u,\mathcal G,\tfrac{\log t}{\sqrt{t}}d_\infty\big)\d u+ \int_0^\infty\sqrt{\log \mathcal N(u,\mathcal G,d_{\G,\tau})}\diff u +\sup_{g\in\mathcal G}\Big(\frac{\log t}{\sqrt{t}}\|g\|_\infty + \lVert g\rVert_{\mathbb{G},\tau}\Big).
\]
If we considered the related discrete time problem of finding uniform moment bounds for additive functionals $\frac{1}{\sqrt{n}} \sum_{k=0}^n g(X_k)$  of a Markov chain $(X_n)_{n \in \N_0}$ and assumed exponential ergodicity of the chain, using the state of the art Bernstein inequality given in \cite[Theorem 6]{adam08} (see also \citep{lema2020}) for the generic chaining procedure would yield an analogous result with an asymptotic version of the variance norm. In particular, the $\log$-scaling of the $\sup$-norm is also present in the discrete time case as a consequence of exponential ergodicity, whereas in the i.i.d.\ case this factor would disappear. Our direct coupling approach therefore yields optimal uniform moment bounds and makes the contribution of the mixing term transparent, which paves the way for studying nonparametric implications of sub-exponential mixing rates for $\sup$-norm estimation problems in continuous time.
\end{remark}


To get a first taste of the consequences of Theorem \ref{Bernstein}, consider the trivial situation where $\cG$ is a singleton set.
This allows the study of rates for the $L^p$-version of von Neumann's ergodic theorem\footnote{Not referring to the $L^p$-statement as Birkhoff's ergodic theorem is not without reason, see \cite{zund2002}.} for continuous time ergodic Markov processes which states that, for $g \in L^{p}(\mu)$,
\[\frac{1}{T} \int_0^T g(X_s) \diff{s} \underset{t  \to \infty}{\longrightarrow} \mu(g), \quad \text{in } L^p(\PP).\]
Indeed, $\beta$-mixing implies strong mixing such that the $\sigma$-algebra of shift invariant sets is $\PP$-trivial and hence the ergodic theorem is satisfied.

\begin{corollary}\label{theo: neumann}
Suppose that $\X$ is exponentially $\beta$-mixing. Then, there exists a constant $C > 0$ such that, for any $T>0$, $1 \leq p < \infty$ and any bounded, measurable function $g$,
\[\Big\lVert \frac{1}{T} \int_0^T g(X_t) \diff{t} - \mu(g) \Big\rVert_{L^p(\PP)} \leq Cp\lVert g \rVert_\infty \frac{1}{\sqrt{T}}.\]
If $\X$ is polynomially mixing of degree $\alpha > 1$, i.e., $\Xi(t) \lesssim t^{-\alpha}$, then for any $p \geq 1$ and $T \geq 4^{(\alpha +p)/\alpha}$ we have
\[\Big\lVert \frac{1}{T} \int_0^T g(X_t) \diff{t} - \mu(g) \Big\rVert_{L^p(\PP)} \lesssim \lVert g \rVert_{\infty} T^{-\big(\tfrac{1}{2} \wedge \tfrac{\alpha}{\alpha +p}\big)}.\]
\end{corollary}

\subsection{Deviation inequalities for suprema of empirical Markov processes}
Theorem \ref{Bernstein} provides a foundation for the derivation of deviation inequalities, as they are needed, for example, for bounding the $\sup$-norm risk of estimators and for the convergence analysis of adaptive estimation procedures.
We will focus on the question of invariant density estimation for Borel right Markov processes, introduced and discussed in Section \ref{sec:basic}.
Recall the definition of $\bm{\Sigma}$ and $\bm{\Theta}$ at the end of that section. 
Given the observation $(X_s)_{0\le s\le T}$, a natural kernel estimator for the invariant density $\rho$ on a domain $D$ of a Markov process $\X \in \bm{\Sigma} \cup \bm{\Theta}$ is given by
\begin{equation}\label{defdensest}
\hat\rho_{h,T}(x)=\frac1T\int_0^TK_h(x-X_s)\d s,\quad x\in\R^d, \quad \text{ where } K_h(\cdot) \coloneqq h^{-d} K(\cdot \slash h), \quad h>0,
\end{equation}
for some smooth, Lipschitz continuous kernel function $K\colon\R^d\to\R$ with compact support $[-1/2,1/2]^d$. 
The knowledge of the invariant density is not only a question of its own interest, but is also needed, among other things, for the implementation of drift estimation procedures or data-driven methods of stochastic control.
Furthermore, this specific estimation problem can be regarded as an acid test for the quality of the statistical analysis:
It is known that the invariant density of (possibly multidimensional) diffusion processes can be estimated with a faster convergence rate than is feasible in the classical discrete i.i.d.\ or weak dependency context.
However, these superior convergence rates can only be verified with sufficiently tight estimates in the proof of the upper bound, more precisely, for the stochastic error part.
Indeed, denoting $\mathbb{H}_{h,T}(x)\coloneqq \hat{\rho}_{h,T}(x)-\E[\hat{\rho}_{h,T}(x)]$, we have the decomposition
\begin{equation} \label{eq: rate0}
\widehat{\rho}_{h,T}(x)-\rho(x)=\mathbb{H}_{h,T}(x)+(\rho \ast K_h-\rho)(x).
\end{equation}
While the bias part is bounded using standard arguments, tight upper bounds on (the supremum of) the stochastic error require specific probabilistic tools.
For bounding the $p$-th uniform moment $\E[\sup_{x\in D}\vert\mathbb H_{h,T}(x)\vert^p]$, we want to apply Theorem \ref{Bernstein} to the function class
\begin{equation}\label{def:G}
\mathcal{G} \coloneqq \big\{\overline{K}((x-\cdot)\slash h): x \in D \cap \mathbb{Q}^d\big\}, \quad \text{ where } \overline{K}((x-\cdot)\slash h) = K((x-\cdot)\slash h)- \mu(K((x-\cdot)\slash h)),
\end{equation}
for some kernel function $K$ with Lipschitz constant $L$ wrt to the $\sup$-norm $\lVert \cdot \rVert_\infty$, and the bandwidth $h$ chosen in $(0,1)$. The following uniform deviation result is central for this purpose. Recall that if $\X \in \bm{\Sigma}$, then by definition, $\X$ is exponentially $\beta$-mixing, i.e., $\beta$-mixing with rate function $\Xi(t) = c_\kappa \mathrm{e}^{-\kappa t}$ for some constants $c_\kappa, \kappa > 0$. 

\begin{lemma}\label{dev:ineq}
Suppose that $\X \in \bm{\Theta} \cup \bm{\Sigma}$ and additionally assume in case $\X \in \bm{\Theta}$ that $\X$ is $\beta$-mixing with strictly decreasing rate function $\Xi(t)$. Then, for any $u_T \geq 1$ such that $\Xi^{-1}(T^{-u_T}) \in \mathsf{o}(T)$ and $T^{-2} \leq h = h_T \in \mathsf{o}(1)$, there exists a constant $c^\ast > 0$ such that for large enough $T$ 
\[\Pro\left(\big\|\hat\rho_{h,T}-\E\hat\rho_{h,T}\big\|_{L^\infty(D)}\geq c^\ast\left( \frac{u_T + \log T}{Th^d} \Xi^{-1}(T^{-u_T}) +T^{-\frac 1 2}\psi_d(h^d)\sqrt{u_T \vee \log(h^{-1})}\right) \right)\leq \e^{-u_T}.\]
In particular, when $\X \in \bm{\Sigma}$, for any $\gamma > 0$ and $u_T \in [1,\gamma \log T]$ there exists a constant $c_\gamma > 0$ such that for large enough $T$
\[\Pro\left(\big\|\hat\rho_{h,T}-\E\hat\rho_{h,T}\big\|_{L^\infty(D)}\geq c_\gamma\Upsilon_{h,T}(u_T)\ \right)\leq \e^{-u_T},\]
where
\begin{equation}\label{def:Ups}
\Upsilon_{h,T}(u)\coloneqq \frac{u(\log T)^2}{Th^d} +T^{-\frac 1 2}\psi_d(h^d)\sqrt{u\vee \log(h^{-1})},\quad u\ge 1.
\end{equation}
\end{lemma}

\section{$\sup$-norm adaptive estimation of the stationary density for general Markov processes} \label{sec: density estimation}
In this section, we demonstrate the effectiveness of our previous results and probabilistic tools in a concrete statistical application.
We already introduced the general form of the kernel invariant density estimator in \eqref{defdensest}. 
In order to quantify the speed of convergence, we will now analyse its convergence behaviour under standard H\"older smoothness assumptions, i.e., we focus on the problem of estimating the invariant density $\rho$ on a domain $D$ of a Markov process $\X \in \bm{\Sigma} \cup \bm{\Theta}$ with $\rho\vert_D \in \mathcal{H}_D(\beta,\mathsf{L})$ (as introduced in \eqref{def:hold}).
For stating our statistical results, we define
\begin{equation}\label{def:raten}
\Phi_{d,\beta}(T)\coloneqq 
\begin{cases} 
1/ \sqrt T, &d=1,\\
\sqrt{\frac{\log T}{T}}, &d=2,\\
T^{-\frac{\beta}{2\beta+d-2}}, &d\ge3,
\end{cases}
\quad \ \text{ and }\quad \ 
\Psi_{d,\beta}(T)\coloneqq
\begin{cases}
\sqrt{\frac{\log T}{T}}, &d=1,\\
\frac{\log T}{\sqrt T}, &d=2,\\
\left(\frac{\log T}{T}\right)^{\frac{\beta}{2\beta+d-2}}, &d\ge3.
\end{cases}
\end{equation}
All proofs of this section are given in Appendix \ref{app: density estimation}. Throughout, $K$ denotes a $\lVert \cdot \rVert_\infty$-Lipschitz kernel of order $\ell$ and with Lipschitz constant $L$ that is supported on $[-1/2,1/2]^d$.

\subsection{General framework}\label{subsec:densgen}

Depending on the concrete application, one might be interested in quantifying the accuracy of estimators in terms of different risk measures.
Our findings from Section \ref{sec:basic} immediately imply an upper bound on the classical mean squared error at some fixed point $x\in\R^d$. 

\begin{corollary}\label{cor:MSE}
Suppose that $\X\in\bm{\Sigma} \cup \bm{\Theta}$.
For $x \in \R^d$ such that there exists an open neighbourhood $D \subset \R^d$ of $x$ such that $\rho\vert_D \in \mathcal{H}_D(\beta, \mathsf{L})$, $\beta \in (0,\ell+1]$, it holds for the kernel estimator
\[
\E\left[\left(\hat \rho_{h,T}(x)-\rho(x)\right)^2\right]\in \mathsf{O}\big(\Phi_{d,\beta}^2(T)\big),\quad \text{ if }
h=h(T)\sim 
\begin{cases} 
T^{-1/\gamma}, & d\le 2,\gamma \in (0,\beta],\\
T^{-1/(2\beta+d-2)}, & d\ge3.
\end{cases}
\]
\end{corollary}
We now turn our focus to the technically significantly more involved problem of $\sup$-norm adaptive invariant density estimation for processes from the class $\bm{\Sigma}$ having H\"older continuous invariant densities. 
We demonstrate that optimal estimation rates in any dimension are achieved by the kernel estimator for a suitable bandwidth choice. 
While in dimension $d=1,2$ the optimal bandwidth has the remarkable property of being independent of the (typically unknown) order $\beta$ of H\"older smoothness, this is not the case in higher dimensions $d \geq 3$. 
In order to remove $\beta$ from the bandwidth choice, we need to find a data-driven substitute for the upper bound on the bias in the balancing process. 
Heuristically, this is the idea behind the Lepski-type selection procedure suggested now:

\begin{enumerate}
\item 
Specify the discrete set of candidate bandwidths
\[
\H_T\equiv \H^{(k)}_T\coloneqq \left\{h_l=\eta^{-l}:\ l\in \N_0,\ \eta^{-l}>\left(\frac{\log_{(k)} T(\log T)^{5}}{T}\right)^{\frac{1}{d+2}}\right\},\quad \eta>1\text{ arbitrary},
\]
for arbitrarily chosen $k \in \N$, and denote by $h_{\min}$ the smallest element in the grid $\H_T$. 
Here, $\log_{(k)}T$ denotes the $k$-th iterated logarithm, iteratively specified by $\log_{(k)}T \coloneqq\log\log_{(k-1)} T$ and $\log_{(0)} T = T$, which is well-defined for $T$ large enough.
\item
Define $\hat h_T\equiv \hat{h}{}^{(k)}_T$ by letting
\begin{equation}\label{est:band0}
\hat h_T\coloneqq\max\left\{h\in\H_T:\left\|\hat\rho_{h,T}-\hat\rho_{g,T}\right\|_{L^\infty(D)}\le\sqrt{\|\hat\rho_{h_{\min},T}\|_{L^\infty(D)}}\sigma(g,T)\ \forall g \leq h,\ g\in\H_T\right\},
\end{equation}
where
\begin{equation}\label{def:sigma}
\sigma(h,T)\coloneqq \frac{\log_{(k)}T(\log T)^2}{Th^d}\log(h^{-1})+\psi_d(h^d)\sqrt{\frac{\log_{(k)}T\log(h^{-1})}{T}}, \quad h\in\H_T.
\end{equation}
\end{enumerate}
Letting $\lVert \cdot \rVert_{L^\infty(D)}$ denote the restriction of the $\sup$-norm to a domain $D\subset \R^d$, we obtain the following result.

\begin{theorem}\label{theo:invdens}
Suppose that $\X\in\bm{\Sigma}$. Let $D\subset \R^d$ be open and bounded. Suppose that $\rho\vert_D \in \mathcal{H}_D(\beta, \mathsf{L})$ with $\beta \in (1,\ell+1]$ for $d = 1$ and $\beta \in (2,\ell+1]$ for $d \geq 2$. Then, for any $p \geq 1$,
\[\left(\E\left[\left\|\hat\rho_{h,T}-\rho\right\|^p_{L^\infty(D)}\right]\right)^{1\slash p} \in \mathsf{O}\big(\Psi_{d,\beta}(T)\big),
\quad\text{ if } h=h(T)\sim \begin{cases}
\log^2 T \slash \sqrt{T}, &d = 1,\\
\log T\slash T^{1\slash 4} , &d = 2,\\
(\log T \slash T)^{1 \slash (2 \beta + d -2)}, &d \geq 3.
\end{cases}\]
For the adaptive bandwidth scheme, let $\hat{h}_T = \hat{h}^{(k)}_T$ be selected according to \eqref{est:band0} for some $k \in \N$. 
Then, if $\rho\vert_D \in \mathcal{H}_D(\beta, \mathsf{L})$ with $\beta \in (2, \ell +1]$, we have in any dimension $d \geq 3$,
\begin{equation}\label{def:adaprate}
	\E\left[\big\|\hat\rho_{\hat{h}_T,T}-\rho\big\|_{L^\infty(D)}\right] \in \mathsf{O}\bigg(\bigg(\frac{\log_{(k)} T \log T}{T}\bigg)^{\frac{\beta}{2\beta + d -2}} \bigg).
\end{equation}
\end{theorem}

The convergence rates introduced in \eqref{def:raten} clearly reflect the fact that the invariant density of stochastic processes can be estimated faster than in the classical context of nonparametric density estimation based on i.i.d.~observations. 
While this is well-known for ergodic continuous diffusion processes (see \citep{dalrei07,strauch2018}), we will show in the following section that the result is fulfilled for a much larger class of stochastic processes.
The additional $\log$-factor occurring in the definition of $\Psi_{d,\beta}(\cdot)$ represents the common price to be paid when switching from the pointwise error control (described by $\Phi_{d,\beta}(\cdot)$) to bounding the $\sup$-norm risk.

\begin{remark}
\begin{enumerate}[label = (\alph*), ref = (\alpha*)]
\item The conditions on the H\"older index $\beta$ stated in Theorem \ref{theo:invdens} are due to two different reasons:
On the one hand, in dimension $d\le2$, we chose a bandwidth not depending on $\beta$ which still achieves the optimal balance between bias and stochastic error.
By choosing a bandwidth dependent on $\beta$ (as in Corollary \ref{cor:MSE}), restrictions on $\beta$ could be avoided. 
However, for the implementation of estimators it is advantageous to be able to choose a bandwidth independent of the typically unknown smoothness $\beta$.
On the other hand, in dimension $d\ge 3$, the assumption on $\beta$ is an unavoidable effect. The coupling error leaves us no other choice but to select the interval block length $m_T$ in the decomposition of \eqref{defdensest} of order $\log T$, which forces $\beta > 2$ to balance out bias and stochastic sensitivity of the estimator. We emphasize that this is not an artifact of our proof strategy since the additional $\log$-factor also appears in the optimal Bernstein inequalities for geometrically ergodic Markov chains in \cite{adam08,lema2020}.  The restriction on $\beta$ can therefore be considered as a price that must be paid for the generality of our exponential $\beta$-mixing assumption. 
\item 
The logarithmic gap (of arbitrary iterative order $k$) between the adaptive rate (see \eqref{def:adaprate}) and the optimal rate $\Psi_{d,\beta}$ in dimension $d \geq 3$ (see \eqref{def:raten}) is \emph{not} a consequence of suboptimality of arguments used in the proof.
Rather, it is a deliberate choice motivated by our desire to introduce a truly adaptive selection procedure that does not rely on the specification of obscure constants. 
To be more precise, a key step in the proof of the upper bound for the adaptive approach requires quantifying the concentration of the estimator $\hat\rho_{h,T}$ around the variance proxy $\sigma(h,T)$ from \eqref{def:sigma}, which is handled with the deviation inequality from Lemma \ref{dev:ineq} involving the term $\Upsilon_{h,T}(\gamma\log T)$ (see \eqref{def:Ups}).
If we remove the factor $\log_{(k)}T$ in the variance proxy $\sigma(h,T)$, we obtain
\[\frac{(\log T)^2}{Th^d}\log(h^{-1})+\psi_d(h^d)\sqrt{\frac{\log(h^{-1})}{T}} \simeq \Upsilon_{h,T}(\gamma\log T).\]
In this case, an exact quantification of the constant $c_\gamma$ from Lemma \ref{dev:ineq} is mandatory, which would then be included as an additional factor in the specification of $\hat{h}_T$ in \eqref{est:band0}. 
Together with an adjustment of the candidate bandwidths $\H_T$, this would allow us to close the logarithmic gap and hence obtain optimal rates for the adaptive procedure. 

However, $c_\gamma$ is of the form $\gamma \times C(D,L,\kappa,c_\kappa,c_2)$---where we recall that $c_\kappa,\kappa$ determine the mixing coefficient and $c_2$ is a constant appearing in the heat kernel bound from Assumption \ref{ass: density bound}---and therefore can only be bounded with explicit knowledge/assumptions on the process. We avoid this fundamental problem in our procedure to not shift the problem from unknown exact smoothness to unknown exact ergodic and small time behaviour, with the price to be paid being a logarithmic loss. In this regard, our approach differs from the bandwidth selection procedure for the $L^2$ risk in \cite{amorino2020invariant}, which relies on the choice of a ``sufficiently large'' constant $k$ that cannot be exactly specified or efficiently chosen in a data-driven way.
\end{enumerate}
\end{remark}

Our previous results rely on the very general conditions \ref{invariant probability measure assumption} and \ref{ass: density bound} as well as assumptions related to the speed of convergence to the invariant distribution, \ref{ass: conv bound} and \ref{ass: beta mix}.
For statistical purposes, however, it is essential to derive results under conditions on the coefficients of the underlying process as easily verifiable as possible.
For this reason, the next two subsections are devoted to investigating specific classes of jump diffusion processes and explicit conditions on their underlying characteristics such that the above assumptions are satisfied and hence statistical conclusions can be drawn from our general theory.

\subsection{Example: L\'evy-driven Ornstein--Uhlenbeck processes}\label{subsec:ou}
As a first example, we discuss estimation rates of $d$-dimensional L\'evy-driven Ornstein--Uhlen-beck processes as representatives of L\'evy-driven jump diffusions with unbounded drift coefficient by establishing assumptions on the characteristics of the L\'evy process that guarantee $\mathbf{X} \in \bm{\Sigma} \cup \bm{\Theta}$. 
Let $\mathbf{Z}$ be a $d$-dimensional L\'evy process with generating triplet $(a,Q,\nu)$, where $a \in \R^d$, $Q \in \R^{d \times d}$ is a symmetric positive semidefinite matrix and $\nu$ is a measure on $\R^d$ satisfying $\nu(\{0\}) = 0$ and $\int_{\R^d} (1\wedge \lVert x \rVert^2)\, \nu(\diff{x}) < \infty$ such that $\E^0[\exp(\mathrm{i}\langle Z_1,\theta\rangle )] = \exp(\psi(\theta))$ with 
\[\psi(\theta) = \mathrm{i}\langle a,\theta \rangle - \frac{1}{2} \langle Q \theta,\theta\rangle + \int_{\R^d \setminus \{0\}} \Big(\mathrm{e}^{\mathrm{i}\langle x,\theta\rangle} - 1 - \mathrm{i} \langle x, \theta \rangle \one_{B(0,1)}(x) \Big)\, \nu(\diff{x}), \quad \theta \in \R^d,\]
where $B(0,1) = \{x \in \R^d: \lVert x \rVert < 1\}$. Then, given some matrix $B \in \R^{d \times d}$, a L\'evy-driven Ornstein--Uhlenbeck process $\mathbf{X}$ is a solution to the SDE 
\[\diff{X_t} = -BX_t\diff{t} + \diff{Z_t},\]
given by 
\[X_t = \mathrm{e}^{-tB}X_0 + \int_0^t \mathrm{e}^{-(t-s)B}\, \diff{Z_s}, \quad t \geq 0.\]
We suppose that the real parts of all eigenvalues of $B$ are positive, implying that $\mathrm{e}^{-tB} \rightarrow \mathbf{0}_{d \times d}$ as $t \to \infty$, and assume the following moment condition 
\begin{equation} \label{eq: ou station}
\int_{\lVert z \rVert > 2} \log \lVert z \rVert \, \nu(\diff{z}) < \infty.
\end{equation}
Then, $\X$ is a Markov process on $\R^d$ with invariant distribution $\mu$ such that 
\begin{align*}
&\phantom{and }\{ \mathscr{F} \mu \}(u) = \exp\Big(\int_0^\infty \psi\big( \mathrm{e}^{-sB^\top}u\big) \diff{s} \Big), \quad u \in \R^d,\\
&\text{and } \varphi_{X_t}^x(u) = \exp\Big(\mathrm{i} \langle x, \mathrm{e}^{-tB^\top}u\rangle + \int_0^t \psi\big( \mathrm{e}^{-sB^\top}u\big) \diff{s} \Big), \quad u,x \in \R^d, t > 0,
\end{align*}
see \cite[Theorem 3.1, Theorem 4.1]{sato84}. Let us now introduce the following conditions.
\begin{enumerate}[leftmargin=*,label=($\mathscr{O}$\arabic*),ref=($\mathscr{O}$\arabic*)]
\item $\mathrm{rank}(Q) = d$;\label{ou2}
\item $\int_{\{\lVert x \rVert > 1\}} \lVert x \rVert^p \, \nu(\diff{x}) < \infty$ for some $p > 0$;\label{ou1}
\item $\int_{\{\lVert x \rVert > 1\}} (\log \lVert x \rVert)^\alpha \, \nu(\diff{x}) < \infty$ for some $\alpha > 2$.\label{ou7}
\end{enumerate}
These assumptions are borrowed from \cite{masuda2004}, \cite{MASUDA200735} and \cite{kevei2018}, where (sub-)exponential ergodicity and exponential $\beta$-mixing of OU-processes are investigated. \ref{ou2} guarantees the strong Feller property of $\X$ and the existence of a $\mathcal{C}_b^\infty$-density for $P_t(x,\cdot)$, $x \in \R^d$ (\cite[Theorem 3.1]{masuda2004}). Similar arguments to the ones in \cite[Theorem 3.2]{masuda2004} also show that under \ref{ou2}, $\mu$ admits a $\mathcal{C}_b^\infty$-density $\rho$. \ref{ou1} and \ref{ou7} are moment assumptions on $\mathbf{Z}$, where \ref{ou7} in absence of \ref{ou1} corresponds to an extremely heavy-tailed distribution and represents a minor strengthening of the necessary and sufficient criterion \eqref{eq: ou station} for stationarity of $\X$.

Based on the results from \cite{masuda2004,MASUDA200735,kevei2018} together with our investigations in Sections \ref{sec:basic} and \ref{sec: density estimation}, we can obtain the following result, which is proved in Appendix \ref{app: density estimation}.

\begin{theorem}\label{theo: ou}
Suppose that \ref{ou2} holds. Then, in any dimension $d \in \N$, \ref{ass: density bound} holds with 
\begin{equation}\label{eq: ou heat}
\sup_{x,y \in \R^d} p_t(x,y) \lesssim t^{-d/2}, \quad t \in (0,1].
\end{equation}
If, additionally,
\begin{enumerate}[label= (\roman*), ref =(\roman*)]
\item \ref{ou1} holds for some $p > 0$, then, for any $d \geq 1$, $\X \in \bm{\Sigma} \cap \bm{\Theta}$;\label{prop: ou2}
\item \ref{ou7} holds, then, for $d=1$, $\X \in \bm{\Theta}$. \label{prop: ou3}
\end{enumerate}
Let $d \geq 1$ in scenario \ref{prop: ou2} and $d=1$ in scenario \ref{prop: ou3}. Then, 
for arbitrary $\beta \in (0, \ell+1]$, we obtain for any $x\in\R^d$ that 
\[
\E\left[\left(\hat \rho_{h,T}(x)-\rho(x)\right)^2\right]\in \mathsf{O}\big(\Phi_{d,\beta}^2(T)\big),\quad \text{ if }
h=h(T)\sim 
\begin{cases} 
T^{-1}, & d\le 2,\\
T^{-1/(2\beta+d-2)}, & d\ge3.
\end{cases}
\]
and for any bounded, open domain $D \subset \R^d$ and $p \geq 1$ that in scenario \ref{prop: ou2}
\[\E\Big[\big \lVert \hat{\rho}_{h,T} - \rho \big\rVert_{L^\infty(D)}^p\Big]^{1/p} \in \mathsf{O}\big(\Psi_{d,\beta}(T)\big),
\quad\text{ if } h=h(T)\sim \begin{cases}
\log^2 T \slash \sqrt{T}, &d = 1,\\
\log T\slash T^{1\slash 4} , &d = 2,\\
(\log T \slash T)^{1 \slash (2 \beta + d -2)}, &d \geq 3.
\end{cases}\]
\end{theorem}
\begin{remark}
\begin{enumerate}[label = (\alph*), ref = (\alpha*)]
\item Since we can choose $\beta > 0$ arbitrarily large, we make the remarkable observation that in the scenarios described above, for any $\varepsilon > 0$ we can obtain the almost superoptimal rates $T^{-(1+\varepsilon)}$ and $(\log T /T)^{1/(2(1+\varepsilon))}$ in any dimension $d \geq 3$ for the pointwise $L^2$ and $\sup$-norm risk, respectively. Moreover, in any dimension, an adaptive choice of the bandwidth is not necessary.
\item The result demonstrates that even under much less stringent assumptions (logarithmic moments and unbounded drift) compared to the class of processes studied in the next section, there are examples of jump diffusions with L\'evy-driven jump part for which optimal estimation results are feasible. It is therefore an interesting question for future research to determine more general coefficient assumptions based on a linear growth condition on the drift that yield optimal estimation properties. 
\end{enumerate}
\end{remark}

\subsection{Example: Non-reversible L\'evy-driven jump diffusion processes}\label{sec: jump diff} 
The goal of this section is to show that solutions of the $d$-dimensional SDE, $d \in \N$,
\begin{equation}\label{SDE}
\begin{split}
X_t=X_0&+\int_0^t b(X_s)\d s+\int_0^t\sigma(X_s)\d W_s+\int_0^t\int_{\R^d}\gamma(X_{s-})z\, \tilde{N}(\diff s,\diff z)
\end{split}
\end{equation}
satisfy assumptions \ref{invariant probability measure assumption}, \ref{ass: density bound} and \ref{ass: beta mix} which then allows using Theorem \ref{theo:invdens} to bound the $\sup$-norm risk of the kernel invariant density estimator. 
Here, $\sigma\colon \R^d\to\R^{d\times d}, \gamma\colon \R^d\to\R^{d\times d}, b\colon \R^d\to\R^d$, $W$ denotes an $\R^d$-valued Brownian motion, $N$ is a Poisson random measure on $[0,\infty)\times\R^d\backslash\{0\}$ with intensity measure $\mu(\mathrm{d}s,\mathrm{d} z)=\d s\,\otimes\,\nu(\mathrm{d} z)$, and $\tilde{N} $ denotes the compensated Poisson random measure. 
Moreover, $\nu$ is a L\'evy measure and we assume that $N,W$ and $X_0$ are independent. 
Note that, if $z \mapsto \gamma(x)z$ is in $L^1(\R^d\backslash\{B_1 \},\nu)$ for all $x\in\R^d, \eqref{SDE}$ is equivalent to
\begin{equation}\label{Compensated SDE}
\begin{split}
X_t=X_0&+\int_0^t b^\ast(X_s)\d s+\int_0^t\sigma(X_s)\d W_s\\&+\int_0^t\int_{\Vert z\Vert\leq 1}\gamma(X_{s-})z\, \tilde{N}(\mathrm{d}s,\mathrm{d} z)+\int_0^t\int_{\Vert z\Vert>1}\gamma(X_{s-})z\, N(\mathrm{d}s,\mathrm{d} z),
\end{split}
\end{equation} with $b^\ast(x)\coloneqq b(x)-\int_{\Vert z\Vert >1}\gamma(x)z\, \nu(dz)$ and $B_1\coloneqq\{z\in\R^d: \Vert z\Vert \leq 1 \}.$
We assume the following.
\begin{enumerate}[leftmargin=*,label=($\mathscr{J}$\arabic*),ref=($\mathscr{J}$\arabic*)]
\item  \label{Lipschitz assumptions}The functions $b,\gamma,\sigma$ are globally Lipschitz continuous, $b$ and $\gamma$ are bounded, and, for $\mathbb{I}_{d\times d}$ denoting the $d\times d$-identity matrix, there exists a constant $c\geq1$ such that \[c^{-1}\mathbb{I}_{d\times d}\leq \sigma \sigma^\top\leq c\mathbb{I}_{d\times d},\]
where the ordering is in the sense of Loewner for positive semi-definite matrices.
\item \label{Kappa Assumptions}$\nu$ is absolutely continuous wrt the Lebesgue measure and, for an $\alpha\in(0,2)$, \[(x,z) \mapsto \Vert \gamma(x)z\Vert^{d+\alpha}\nu(z)\]
is bounded and measurable, where, by abuse of notation, we denoted the density of $\nu$ also by $\nu$. Furthermore,
if $\alpha=1$, \[\int_{r<\Vert\gamma(x) z\Vert\leq R} \gamma(x)z\,\nu(\mathrm{d}z)=0,\quad\text{ for any } 0<r<R<\infty,\ x\in\R^d.\] 
\item \label{Ergodicity Assumptions} There exist $c_1,c_2>0$ and $\eta_0 >0$ such that
\[\qv{x,b(x)}\leq -c_1\Vert x\Vert, \quad \forall x:\Vert x\Vert\geq c_2,\quad\text{ and }\quad
\int_{ \R^d}\Vert z\Vert^2 \e^{\eta_0\Vert z\Vert}\nu(\mathrm{d}z)<\infty.\]
\end{enumerate}
In \citep{amorino2020invariant}, the authors also investigate $L^2$ invariant density estimation for jump diffusions and use a similar approach for formulating requirements on the diffusion coefficients which imply their respective heat kernel bound and mixture assumptions. 
The conditions however are more restrictive and, in particular, the case of continuous diffusions cannot be handled within their framework since it requires $\supp(\nu)=\R^d$ and $\operatorname{det}(\gamma(x))>c$ for some constant $c>0$ and all $x \in\R^d$. In \cite{amorino21}, the authors improve the $L^2$ rate for dimension $d = 1$ from \cite{amorino2020invariant} to the parametric rate $1/T$ by imposing an additional smoothness restriction on the jump measure. Our main contribution in this section is to show that under the less stringent assumptions above, optimal convergence rates can be achieved not only wrt the $L^2$ risk but even wrt $\sup$-norm risk in any dimension.

Note that \ref{Lipschitz assumptions} and \ref{Ergodicity Assumptions} directly imply $\gamma(x)z\in L^1(\R^d\backslash\{B_1 \},\nu)$, so \eqref{SDE} and \eqref{Compensated SDE} are equivalent. The subsequent lemma shows that, under the given assumptions, there exists a pathwise unique strong solution for \eqref{SDE} and that the conditions of Corollary 1.5 of \citep{CHEN20176576} hold, implying the heat kernel bound \eqref{transition density bounds}. All proofs can be found in the Appendix.
\begin{lemma}\label{Chen and Applebaum assumptions}
Let \ref{Lipschitz assumptions}--\ref{Ergodicity Assumptions} hold. 
Then, \eqref{SDE} admits a c\`adl\`ag, non-ex\-plo\-sive, pathwise unique, strong solution possessing the strong Markov property, and the assumptions ($\mathbf{H}^\alpha$) and ($\mathbf{H}^\kappa$) of \citep{CHEN20176576} hold.
\end{lemma}
Let $\X$ be the unique solution of \eqref{SDE} described in Lemma \ref{Chen and Applebaum assumptions}.

\begin{corollary}\label{coroll: heat jump}
Let \ref{Lipschitz assumptions}--\ref{Ergodicity Assumptions} hold. Then, transition densities $(p_t)_{t > 0}$ exist and there are constants $C,\lambda>1 $ such that the solution $\X$ of \eqref{SDE} satisfies the following heat kernel estimate
for all $x,y\in\R^d,0<t\leq 1,$
\begin{equation}\label{transition density bounds}
\begin{split}
&C^{-1}(t^{-d/2}\exp(-\lambda\Vert x-y\Vert^2/t)+(\operatorname{inf}_{x\in\R^d}\operatorname{ess\,inf}_{z\in\R^d}\kappa_\alpha(x,z))t(\Vert x-y\Vert +t^{1/2})^{-d-\alpha})\\
&\quad\leq p_t(x,y)\leq C(t^{-d/2}\exp(-\Vert x-y\Vert^2/(\lambda t))+\Vert \kappa_\alpha\Vert_\infty t(\Vert x-y\Vert +t^{1/2})^{-d-\alpha}),
\end{split}
\end{equation}
where $\kappa_\alpha(x,z)=\Vert \gamma(x)z\Vert^{d+\alpha}\nu(z)$. In particular, assumption \ref{ass: density bound} is satisfied.
\end{corollary}
Now our goal is to show that the solution $\X$ of \eqref{SDE} fulfills the fundamental assumption \ref{ass: bounded invariant} and exponential ergodicity along with the mixing property \ref{ass: beta mix}. First, observe that \ref{Lipschitz assumptions} implies that $b \in \mathcal{C}_b(\R^d; \R^{d})$ and $\sigma,\gamma \in  \mathcal{C}_b(\R^d; \R^{d \times d})$ and hence Theorem 6.7.4 in \citep{applebaum09} guarantees that the unique c\`adl\`ag Markov process $\X$ solving \eqref{SDE} is Feller and therefore Borel right.  Further, Corollary \ref{coroll: heat jump} in particular implies the existence of bounded transition densities and thus, by Lemma \ref{lemma: invariant bounded}, it suffices to show the existence of an invariant distribution. This will be done as a byproduct while proving exponential ergodicity and the exponential mixing property \ref{ass: beta mix}. For this, we will employ results of Masuda \citep{MASUDA200735} which are again based on the theory of stability of continuous-time Markov processes of Meyn and Tweedie \citep{meyn_tweedie_1993}. These lead us to the following proposition.
\begin{proposition}\label{Exponential Ergodicity}
Grant assumptions \ref{Lipschitz assumptions}--\ref{Ergodicity Assumptions}. 
Then, an invariant distribution exists, $\X$ is $V$-exponentially ergodic with locally bounded $V$ and the process $\X$ started in the invariant distribution $\mu$ is exponentially $\beta$-mixing.
\end{proposition}
Gathering the results of Corollary \ref{coroll: heat jump} and Proposition \ref{Exponential Ergodicity} and employing Lemma \ref{prop: poisson} now yields that \ref{invariant probability measure assumption}--\ref{ass: conv bound} and \ref{ass: beta mix} are fulfilled for the solution $\mathbf{X}$ of \eqref{SDE}, i.e., $\mathbf{X} \in \bm{\Sigma} \cap \bm{\Theta}$. 
In particular, the results from Section \ref{subsec:densgen} can be applied.
\begin{theorem}\label{theo: ljd}
Let $D \subset \R^d$ be open and bounded and assume \ref{Lipschitz assumptions}--\ref{Ergodicity Assumptions}. 
If $\rho\vert_D \in \mathcal{H}_D(\beta,\mathsf{L})$ with $\beta \in (1,\ell+1]$ for $d=1$ and $\beta \in (2,\ell+1]$ for $d \geq 2$, then, the $\sup$-norm risk of the kernel estimator defined in \eqref{defdensest} is of order 
\[\E\Big[\lVert \hat{\rho}_{h,T} - \rho \rVert^p_{L^\infty(D)} \Big]^{1\slash p} \in 
\begin{cases}
\mathsf{O}(\sqrt{\log T/T}), &d=1,\\
\mathsf{O}(\log T/\sqrt T), &d=2,\\
\mathsf{O}\big((\log T/T)^{\beta/(2\beta+d-2)}\big), &d\ge3,\end{cases}
\quad \text{ if }
h \sim  \begin{cases} 
\log^2 T \slash \sqrt{T}, & d = 1,\\
\log T \slash T^{1/4}, & d = 2,\\
(\log T/T)^{-1/(2\beta+d-2)}, & d\ge3.
\end{cases}
\]
for any $p \geq 1$. If $\hat{h}_T \equiv \hat{h}^{(k)}_T$ is chosen adaptively according to \eqref{est:band0} for some $k \in \N$, then for any $d \geq 3$,
\[\E\Big[\big\lVert \hat{\rho}_{\hat{h}_T} - \rho \big\rVert_{L^\infty(D)} \Big] \in \mathsf{O}\left(\bigg(\frac{\log_{(k)} T \log T}{T}\bigg)^{\beta/(2\beta + d -2)} \right).\]
Moreover, for any $x \in \R^d$ such that $\rho\vert_D \in \mathcal{H}_D(\beta,\mathsf{L})$ for some $\beta \in (0,\ell+1]$ and a neighborhood $D$ of $x$, we have the pointwise $L^2$ risk estimate 
\[
\E\left[\left(\hat \rho_{h,T}(x)-\rho(x)\right)^2\right]\in \begin{cases}
\mathsf{O}(1/T), &d=1,\\
\mathsf{O}(\log T/T), &d=2,\\
\mathsf{O}\big(T^{-2\beta/(2\beta+d-2)}\big), &d\ge3,\end{cases} \quad \text{ if }
h \sim  \begin{cases} 
T^{-1/\gamma}, & d \leq 2, \gamma \leq \beta,\\
T^{-1/(2\beta+d-2)}, & d\ge3.
\end{cases}\]
\end{theorem}



\appendix
\section{Supplements of Section \ref{sec:basic}}
\subsection{Assumption \ref{ass: conv bound} and the exponential $\beta$-mixing property} \label{sec:betamix}

As in the rest of the paper, we will assume in this section that $\X$ is a Borel right Markov process with unique invariant distribution $\mu$ possessing a Lebesgue density $\rho$. Let us start by collecting some important definitions in the realm of stability theory of Markov processes. We say that $\mathbf{X}$ is $\psi$-\textit{irreducible} for some $\sigma$-finite measure $\psi$ on its state space if $\psi(B) > 0$ for some Borel set $B$ implies
\[U(x,B) \coloneqq \int_0^\infty P_t(x,B) \diff{t} = \E^x[\eta_B]> 0\]
for any $x \in \R^d$, i.e., the expected sojourn time $\eta_B$ of $\mathbf{X}$ in $B$ (or, equivalently, the potential of $B$), where $\eta_B = \int_0^\infty \one_{\{X_t \in B\}} \diff{t}$, when $\mathbf{X}$ is started in an arbitrary state is strictly positive. If for $B \in \mathcal{B}(\R^d)$, $\psi(B) > 0$ even implies $\PP^x(\eta_B = \infty) = 1$ for any $x \in \R^d$, we say that $\mathbf{X}$ is \textit{Harris recurrent} and that $\psi$ is a Harris measure. Harris recurrent Markov processes having an invariant distribution (which is unique in this case) are called positive Harris recurrent.  A Borel set $C$ is called \textit{small} if there exists $T > 0$ and a non-trivial measure $\nu$ on the state space such that
\[P_T(x,\cdot) \geq \nu(\cdot), \quad x \in C.\]
\textit{Petite} sets generalize the notion of small sets. We call a Borel set $C$ petite if there exists a sampling distribution $a$ on $(\R_+,\mathcal{B}(\R_+))$ and a non-trivial measure $\nu_a$ on the state space s.t.\
\[K_a(x ,\cdot) \coloneqq \int_0^\infty P_t(x,\cdot)\, a(\diff{t}) \geq \nu_a(\cdot), \quad x \in C,\]
i.e., small sets are petite sets with sampling distribution $a = \delta_T$ for some $T>0$. All three concepts have obvious counterparts for discrete-time chains. If moreover the $\psi$-irreducible process $\mathbf{X}$ possesses a small set $C$ such that $\psi(C) > 0$ and there is $T> 0$ such that
$P_t(x,C) > 0$, $\forall x \in C$, $t \geq T$, we say that $\mathbf{X}$ is \textit{aperiodic}.

These notions are of central importance in the theory of stability of Markovian processes on general state spaces in both discrete as well as continuous time. 
In discrete time, the existence of small sets allows the construction of a related Markov chain via the technique of Nummelin splitting, which shares the same stability properties with the original chain but possesses an atom.
This in turn allows to transfer well-known reasoning in Markov chain theory on countable state spaces to the general state space situation with renewal arguments. 
With the Meyn and Tweedie approach to stability of continuous-time Markov processes, which heavily involves the aforementioned concept of aperiodicity, we can then infer stability properties through sampled chains, generalizing discrete-time results to continuous time. 
For a complete picture in discrete time, we refer to the monograph \citep{MeynTweedie2009}. 
Continuous-time theory was developed in the 1990s in a series of papers \citep{DownMeynTweedie1995,MeynTweedie1993b,MeynTweedie1993,meyn_tweedie_1993} and many other subsequent contributions.

We see that these concepts are quite natural when we aim to infer stability of general Markov processes, and we need no more than irreducibility as well as the property that compact sets are small together with exponential decay in \eqref{cond2:uni} to infer exponential $\beta$-mixing of the process.

\begin{proposition} \label{prop:betamix}
Suppose that $\mathbf{X}$ is $\psi$-irreducible and that every compact set $\mathcal{S} \subset \R^d$ is small. Moreover, let \ref{ass: conv bound} be satisfied with $\fset = \R^d$ and
\begin{equation} \label{cond3:uni}
r_{\mathcal{S}}(t) \coloneqq C_{\mathcal{S}} \mathrm{e}^{-\kappa_{\mathcal{S}}t},\quad t > 0,
\end{equation}
with constants $C_{\mathcal{S}}, \kappa_{\mathcal{S}} > 0$. Then, $\mathbf{X}$ is exponentially $\beta$-mixing.
\end{proposition}
\begin{proof}
Let $\mathcal{S} \subset \R^d$ be compact such that $\lebesgue(\mathcal{S}) > 0$. Since $\R^d$ can be covered by countably many compact sets and the irreducibility measure $\psi$ is $\sigma$-finite, we can also assume that $\psi(\mathcal{S})> 0$ and $\mu(\mathcal{S}) > 0$. Letting $(P_t)_{t \geq 0}$ denote the semigroup associated to $\mathbf{X}$, we obtain from \eqref{cond2:uni} and \eqref{cond3:uni} that, for any $x \in \mathcal{S}$ and $t > 0$,
\begin{equation*}
\lvert P_t(x,\mathcal{S}) - \mu(\mathcal{S})\rvert \leq \int_{\mathcal{S}} \lvert p_t(x,y) - \rho(y)\rvert \diff{y} \leq C_{\mathcal{S}} \mathrm{e}^{-\kappa_{\mathcal{S}}t} \lebesgue(\mathcal{S}) = \tilde{C}_{\mathcal{S}}\mathrm{e}^{-\kappa_{\mathcal{S}}t},
\end{equation*}
with $\tilde{C}_{\mathcal{S}} = C_{\mathcal{S}}\lebesgue(\mathcal{S}).$ Since $\mu(\mathcal{S}) > 0$, this implies in particular that there exists $T(\mathcal{S}) > 0$ such that $P_t(x,\mathcal{S}) > 0$ for all $t \geq T(\mathcal{S})$ and $x \in \mathcal{S}$. Since $\mathcal{S}$ is small by assumption, it follows that $\mathbf{X}$ is aperiodic. Hence, by Theorem 5.3 in \citep{DownMeynTweedie1995} and the remarks thereafter, there exists (a) an extended real-valued measurable function $V \geq 1$ such that, for some $T > 0$, we have
\begin{equation} \label{drift skeleton}
P_T V(x) \leq \lambda V(x) + b \mathbf{1}_\Theta
\end{equation}
for some $0 < \lambda < 1$, $ b\geq 0$ and a small set $\Theta \in \mathcal{B}(\R^d)$ and (b) a set $S_V \subset \{V < \infty\}$, which is full and absorbing\textemdash that is, $\mu(S_V) = 1$ and $P_T(x, S_V) = 1$ for any $x \in S_V$\textemdash such that $\mathbf{X}$ restricted to $S_V$ is exponentially ergodic in the sense
\begin{equation} \label{exponential ergodicity full}
\lVert P_t(x,\cdot) - \mu \rVert_{\mathrm{TV}} \leq CV(x)\mathrm{e}^{-\kappa t}, \quad x \in S_V,
\end{equation}
for some constants $C,\kappa > 0$. Noting that \eqref{drift skeleton} implies
\[\Delta \tilde{V} \leq -V + \frac{b}{1-\lambda}\mathbf{1}_\Theta\]
with $\tilde{V} = V\slash (1-\lambda) \geq 0$ and $\Delta \coloneqq P_T - \mathbb{I}$, it follows from Theorem 14.0.1 in \citep{MeynTweedie2009} that $\mu(V) < \infty.$ The claim on exponential $\beta$-mixing of the process now follows from \eqref{exponential ergodicity full} since
\[
\int_{\R^d} \lVert P_t(x,\cdot) - \mu\rVert_{\mathrm{TV}}\, \mu(\diff{x}) = \int_{S_V} \lVert P_t(x,\cdot) - \mu\rVert_{\mathrm{TV}}\, \mu(\diff{x})
\leq C\mathrm{e}^{-\kappa t} \int_{S_V} V(x)\, \mu(\diff{x})= \tilde{C} \mathrm{e}^{-\kappa t},
\]
for any $t \geq 0$, where finiteness of $\tilde{C} = C \mu(V)$ was discussed above and for the first equality we used that $S_V$ is full. 
\end{proof}

Compactness of small sets can be inferred for a quite general class of Markov processes. We say that $\mathbf{X}$ is a $T$-process if there exists a non-trivial continuous component for some sampled chain, i.e., there exists a sampling distribution $a$ on $(\R_+,\mathcal{B}(\R_+))$ and a non-trivial, lower semi-continuous kernel $T$ on the state space s.t.\
\[K_a(x,\cdot) \geq T(x,\cdot), \quad x \in \R^d.\]
Many processes in applied probability can be shown to be $T$-processes such as price processes driven by L\'evy risk and return processes \citep{paulsen98}, certain piecewise deterministic Markov processes used for MCMC \citep{bierkens2019} or queuing networks \citep{down97}. Moreover, any open set irreducible weak $\mathcal{C}_b$-Feller process is a $T$-process (cf.\ \citep[Theorem 7.1]{tweedie1994}). Markov processes having the strong Feller property\textemdash that is, the semigroup satisfies $P_t \mathcal{B}_b(\R^d) \subset \mathcal{C}_b(\R^d)$ for all $t \geq 0$\textemdash are trivially $T$-processes, since any operator $P_t$ is a continuous component for itself. Here, we denoted  by $\mathcal{C}_b(\R^d)$ the family of bounded, continuous functions on $\R^d$  and by $\mathcal{B}_b(\R^d)$ the family of bounded Borel functions on $\R^d$. The strength of Markov processes with the strong Feller property---and $T$-processes as a generalization of such processes---comes from making possible to connect distributional properties of the Markov process induced by the semigroup and topological properties of the state space, thus allowing to use knowledge of the topology to infer strong stability results of the Markov process.  Classical examples of Markov processes with the strong Feller property are L\'evy processes with absolutely continuous semigroup with respect to the Lebesgue measure \citep[Theorem 2.2]{Hawkes1979}, diffusion processes with hypoelliptic Fisk--Stratonovich-type generator  \citep[Lemma 5.1]{ichi1974}, diffusion processes on Hilbert spaces under appropriate assumptions on the coefficients \citep[Theorem 1.2]{peszat1995}, or solutions of different classes of parabolic SPDEs \citep{maslowski89,prato91,prato95,eckmann2001}. More recently, the strong Feller property was discussed for switching (jump-)diffusions \citep{xi2013,zhu2009}, for jump-diffusions with non-Lipschitz coefficients \citep{xi2019}, or Markov semigroups generated by singular SPDEs such as the KPZ equation in Hairer and Mattingly \citep{hairer2016}. For an account discussing conditions for which (weak) $\mathcal{C}_b$-Feller processes 
are even strong Feller, we refer to Schilling and Wang \citep{schilling2012}.

Let us now infer the exponential $\beta$-mixing property for $T$-processes given exponential decay in \eqref{cond2:uni} and, as a natural mixing requirement, ergodicity in the sense of total variation convergence to the invariant distribution, i.e.,
$\lVert P_t(x,\cdot) - \mu \rVert_{\mathrm{TV}} \underset{t \to \infty}{\longrightarrow} 0$, $\forall x \in \R^d$. Note that indeed, dominated convergence shows that any stationary, ergodic Markov process is $\beta$-mixing.

\begin{proposition} \label{prop2:betamix}
Let $\mathbf{X}$ be an ergodic $T$-process such that \ref{ass: conv bound} is satisfied for $r_{\mathcal{S}}$ given as in \eqref{cond3:uni} and $\fset = \R^d$. Then, $\mathbf{X}$ is positive Harris recurrent, every compact set is small and $\mathbf{X}$ is exponentially $\beta$-mixing.
\end{proposition}
\begin{proof}
For the exponential $\beta$-mixing property, it suffices to check that every compact set is small by Proposition \ref{prop:betamix}, since ergodicity clearly implies $\mu$-irreducibility of $\X$. We prove this property together with positive Harris recurrence at once. To this end, for a given $\varepsilon > 0$, choose a compact set $C \subset \R^d$ such that $\mu(C) \geq 1 - \varepsilon$. Then, for fixed $x \in \R^d$, ergodicity guarantees that $\lim_{t \to \infty} \PP^x(X_t \in C) \geq 1 - \varepsilon$, and hence $\X$ is bounded in probability on average as defined on p.\ 495 of \citep{MeynTweedie1993}. Since $\X$ is an irreducible $T$-process, Theorem 3.2 and Theorem 4.1 of the same reference yield Harris recurrence and petiteness of compact sets.
It remains to show that small and petite sets coincide for the given process. The reverse implication of Theorem 6.1 in \citep{MeynTweedie1993} guarantees that there exists an irreducible skeleton $\X^\Delta = (X_{n \Delta})_{n \in \N_0}$ for some $\Delta > 0$ thanks to ergodicity and positive Harris recurrence of $\mathbf{X}$. Proposition 6.1 in \citep{MeynTweedie1993} therefore implies equivalence of small and petite sets, which finishes the proof. 
\end{proof}

\subsection{Proofs for Section \ref{sec:basic}}\label{app: basic}
\begin{proof}[Proof of Proposition \ref{prop:varextra}]
Without loss of generality, let $T\geq1$ be fixed.
Then, using the Markov property and the invariance of $\mu$, for any $\delta\in[0,1]$,
\begin{align*}
&\Var\left(\int_0^Tf(X_s)\diff s\right)\ =\ \E\left[\left(\int_0^T(f(X_s)-\E f(X_0))\diff s\right)^2\right]\\
&\hspace*{3em}=\ 2\E\left[\int_0^T\int_0^u\left(f(X_0)-\E f(X_0)\right)\left(f(X_{u-s})-\E f(X_0)\right)\diff s\diff u\right]\\
&\hspace*{3em}=\ 2\int_0^T\int_0^u\left(\E\left[f(X_0)f(X_{u-s})\right]-\left(\E f(X_0)\right)^2\right)\diff s\diff u\\
&\hspace*{3em}=\ 2\int_0^T\int_0^u\left[\iint_{\R^{d\times d}}f(x)f(y)p_{u-s}(x,y)\diff y\mu(\diff x)-\int f(x)\,\mu(\diff x)\int f(y)\rho(y)\diff y\right]\diff s\diff u\\
&\hspace*{3em}=\ 2\int_0^T\int_0^u\int_\fset \int_{\R^d}f(x)f(y)\left(p_{u-s}(x,y)-\rho(y)\right)\diff y\,\mu(\diff x)\diff s\diff u \\
&\hspace*{3em}=\ 2\left(\mathcal I(0,\delta)+\mathcal I(\delta,1)+ \mathcal I(1,T)\right),
\end{align*}
with (substituting $v=u-s$)
\[
\mathcal I(a,b)\coloneqq\ \int_a^b(T-v)\int_{\R^d}\int_{\fset}f(x)f(y)\left(p_v(x,y)-\rho(y)\right)\,\mu(\diff x)\diff y\diff v,\quad 0\le a<b\le T.
\]
It follows from the assumption on the convergence of the transition density in \eqref{cond2:uni} that
\begin{align*}
\mathcal I(1,T)&\leq\
\int_1^T(T-v)\sup_{x \in \cS \cap \fset ,y\in\mathcal S}\left|p_v(x,y)-\rho(y)\right|\diff v\ \iint_{\R^{d}\times \R^d}f(x)f(y)\,\mu(\diff x)\diff y\\
&\leq\ T \|f\|_\infty^2\lebesgue(\mathcal S)\mu(\mathcal S)\int_1^Tr_{\mathcal{S}}(v)\diff v\ \leq\ c_{\mathcal{S}}T\|f\|_\infty^2\lebesgue(\mathcal S)\mu(\mathcal S).
\end{align*}
It remains to consider the first parts of the integral.
We now restrict to dimension $d\ge3$; the remaining cases are handled with analogous arguments.
Note first that
\begin{equation}\label{int0}
\mathcal I(0,\delta)\
\leq\ T\|f\|_\infty^2\int_0^\delta\iint_{\mathcal{S} \times \R^d} p_v(x,y)\,\mu(\diff x)\diff y\diff v\ =\ T\|f\|_\infty^2\mu(\mathcal{S})\delta.
\end{equation}
On the other hand, the heat kernel bound \eqref{cond1:uni} gives for any  $x,y \in \R^d$,
\[\int_{\delta}^1 p_v(x,y) \diff v \leq c_2 \int_{\delta}^1 v^{-d\slash2} \diff v = c_2^\prime \delta^{1-d\slash2},\]
where $c_2'=2\slash (d-2) c_2$. Letting $\delta=(\lebesgue(S))^{2/d}$ and exploiting that $\lebesgue(\mathcal{S})< 1$, it follows
\[
\mathcal I(\delta,1)\ \le\ T\|f\|_\infty^2\int_\delta^1\iint_{\mathcal{S}^2} p_v(x,y)\,\mu(\diff x)\diff y\diff v\ \le\ c_2'T\|f\|_\infty^2\mu(\mathcal S)(\lebesgue(\mathcal S))^{\frac2d}.
\] 
\end{proof}

\begin{proof}[Proof of Lemma \ref{prop: poisson}]
By the semigroup property of $(P_t)_{t \geq 0}$ and invariance of $\mu$ we have for any $ t > 1$ and $y \in \R^d$ and $\mu$-a.e.\ $x \in \R^d$,
\begin{align*}
\lvert p_t(x,y) - \rho(y) \rvert &\leq \int_{\R^d} p_1(z,y) \lvert p_{t-1}(x,z) - \rho(z) \rvert \diff{z}\\
&\leq \lVert p_1 \rVert_\infty \int_{\R^d} \lvert p_{t-1}(x,z) - \rho(z) \rvert \diff{z} \\
&= 2\lVert p_1 \rVert_\infty \lVert P_{t-1}(x, \cdot) - \mu\rVert_{\mathrm{TV}} \leq 2\lVert p_1 \rVert_\infty CV(x) \Xi(t-1),
\end{align*}
where the equality follows from Scheff\'{e}'s theorem, see \cite[Lemma 2.1]{tsy09}. Thus for any compact set $\mathcal{S}$  and $r_{\mathcal{S}}(t) = 2C\lVert p_1 \rVert_\infty \rVert_\infty)\sup_{x \in \mathcal{S} \cap \fset} V(x) \Xi(t-1)$ it follows that
\[\int_1^\infty \sup_{x \in \mathcal{S} \cap \fset,y \in \mathcal{S}} \lvert p_t(x,y) -\rho(y)\rvert \diff{t} \leq \int_1^\infty r_{\mathcal{S}}(t) \diff{t}\lesssim \sup_{x \in \mathcal{S} \cap \fset} V(x) \int_0^\infty \Xi(t) < \infty,\]
by local boundedness of $V\one_\fset$ and the convergence assumption on $\Xi$, which yields \ref{ass: conv bound}.
\end{proof}

\begin{proof}[Proof of Lemma \ref{lemma: invariant bounded}]
Let $B\in\mathcal{B}(\R^d)$ such that $\lebesgue(B)=0$. 
Then, it holds that
\begin{align*}
\mu(B)=\int_{\R^d}\int_B p_\Delta(x,y)\d y\,\mu(\mathrm{d}x)=0,
\end{align*}
which yields the existence of a Lebesgue density $\rho$ of $\mu$ by the Radon--Nikodym theorem. 
Now, let $B\in\mathcal{B}(\R^d)$ such that $\lebesgue(B)>0$. 
Arguing as above and using boundedness of $p_\Delta$, we get
\begin{align*}
\frac{\int_B \rho(x) \,\lebesgue(\mathrm{d}x)}{\lebesgue(B)}\leq c.
\end{align*}
Now the Lebesgue differentiation theorem yields $\operatorname{ess\,sup}\rho\leq c,$ and defining 
\[\rho_b(x)=\rho(x)\1_{[0,c]}(\rho(x)), \quad x \in \R^d,\] 
we have $\rho=\rho_b$ almost everywhere and $\rho_b\leq c$ which completes the proof. 
\end{proof}

\begin{proof}[Proof of Proposition \ref{prop:varmulti}]
Let $0<\delta<1\leq D$.
Analogously to the proof of Proposition \ref{prop:varextra}, one can compute that
\begin{align*}
\operatorname{Var}\left(\int_0^T f(X_t)\diff t \right)
&=2\int_0^T(T-v)\iint_{\R^{d\times d}} f(x)f(y)(p_v(x,y)-\rho(y))\,\mu(\diff x)\diff y\diff v\\
&\leq 2T\Vert f\Vert_\infty^2\Big(\int_0^D\iint_{\mathcal{S}^2} p_v(x,y)\,\mu(\diff x)\diff y\diff v+\int_D^T\iint_{\cS^2}(p_v(x,y)-\rho(y))\,\mu(\diff x)\diff y\diff v\Big)\\
&= 2T\Vert f\Vert_\infty^2 (\II_\delta+\II_D+\II_T),
\end{align*}
where $\II_\delta\coloneqq \int_0^\delta\iint_{\mathcal{S}^2} p_v(x,y)\mu(\diff x)\diff y\diff v$, $\II_D\coloneqq \int_\delta^D\iint_{\mathcal{S}^2} p_v(x,y)\mu(\diff x)\diff y\diff v$ and
\[\II_T\ \coloneqq\ \int_D^T\int_{\cS}(P_v(x,\cS)-\mu(\cS))\, \mu(\diff x)\diff v.\]
As before (see \eqref{int0}) and under our additional assumption that $\rho$ is bounded, it holds
\begin{equation}\label{eq: int delta}
\II_\delta \le \mu(\mathcal{S})\delta \leq \lVert\rho \rVert_\infty \lebesgue(\mathcal{S})\delta.
\end{equation}
Furthermore, exploiting the mixing property of $\X$,
\begin{equation} \label{eq: int mix}
\II_T\le\ \int_D^T\int\|P_v(x,\cdot)-\mu(\cdot)\|_{\TV}\mu(\diff x)\diff v\ \le\ c_\kappa\int_D^T\e^{-\kappa v}\diff v\ \le\ \frac{c_\kappa}{\kappa}\e^{-\kappa D}\mathbf{1}_{(D,\infty)}(T).
\end{equation}
By assumption \ref{ass: density bound},
$p_v(x,y)\leq c_2v^{-d/2},$
for $0<t\leq1$. 
Hence, we have $p_{1/2}(x,y)\leq c_2 2^{d/2}\eqqcolon c_{p}$ which implies
\[p_t(x,y)=\int p_{t-1/2}(x,z)p_{1/2}(z,y)\d z \leq c_p,\]
for all $t>1/2$. Since $\delta<1\leq D$, it follows
\[
\int_\delta^Dp_v(x,y)\d v
\leq c_2\int_\delta^1v^{-d/2}\d v+c_pD\1_{(1,\infty)}(D)
\leq c_{\delta,D}\Big(\int_\delta^1v^{-d/2}\d v+D\1_{(1,\infty)}(D)\Big)
\]
for $c_{\delta,D}\coloneqq c_2+ c_p$. For $d\geq 3$, this implies
\begin{equation}
\begin{split}\label{I delta D computation 3+}
\int_\delta^Dp_v(x,y)\d v&\leq c_{\delta,D}\Big(\int_\delta^1v^{-d/2}\d v+D\1_{(1,\infty)}(D)\Big)\\
&\leq c_{\delta,D}\Big((d/2-1)^{-1}\delta^{1-d/2}+D\1_{(1,\infty)}(D)\Big)
\\
&\leq c^\prime_{\delta,D}\Big(\delta^{1-d/2}+D\1_{(1,\infty)}(D)\Big),
\end{split}
\end{equation}
where $c^\prime_{\delta,D}\coloneqq 2c_{\delta,D}$.
Letting $\delta=\lebesgue(\cS)^{2/d}$, $D=(1\lor -\tfrac{2}{\kappa}\log(\lebesgue(\cS)))\land T$, \eqref{I delta D computation 3+} and $\lebesgue(\cS)<1$ imply
\begin{align*}
\int_\delta^Dp_v(x,y)\d v
&\leq c^\prime_{\delta,D}\Big( \lebesgue(\cS)^{2/d-1}+\tfrac{2}{\kappa}\log(\lebesgue(\cS)^{-1}) \Big)
\leq c^\prime_{\delta,D}\Big( \lebesgue(\cS)^{2/d-1}+\tfrac{2}{\kappa(1-2/d)}\lebesgue(\cS)^{2/d-1} \Big)\\
&\leq c^{\prime\prime}_{\delta,D}\lebesgue(\cS)^{2/d-1},\quad\text{ for }c^{\prime\prime}_{\delta,D}\coloneqq c^\prime_{\delta,D}(1+\tfrac{2}{\kappa(1-2/d)}),
\end{align*}
where we have used the well-known inequality $\log x\le nx^{1/n},\, x,n>0$. 
Using Fubini's theorem, this directly implies
\begin{equation}\label{I D bound 3+}
\mathcal{I}_D= \int_\delta^D\iint_{\cS^2} p_v(x,y)\mu(\mathrm{d}x)\d y\d v
\leq  c^{\prime\prime}_{\delta,D}\mu(\cS)\lebesgue(\cS)^{2/d} \leq c^{\prime\prime}_{\delta,D}\lVert \rho \rVert_\infty\lebesgue(\cS)^{2/d+1}
\end{equation}
for $d\geq3$. 
Noting that our choice of $\delta$ and $D$ implies by \eqref{eq: int delta} and \eqref{eq: int mix} that
\[\mathcal{I}_\delta \leq \lVert\rho\rVert_\infty\lebesgue(\cS)^{2\slash d+1}\quad\text{ and }\quad \mathcal{I}_T \leq \frac{c_\kappa}{\kappa} \lebesgue(\cS)^2 \leq \frac{c_\kappa}{\kappa} \lebesgue(\cS)^{2\slash d+1},\]
\eqref{eq: var mix} follows for any $d\geq 3$ by combining these estimates with \eqref{I D bound 3+}.
The case $d=2$ is treated by similar arguments. 
\end{proof}

\section{Proofs for Section \ref{sec: deviation}} \label{app: deviation}
\begin{proof}[Proof of Theorem \ref{Bernstein}]
We start by splitting the process $(X_s)_{0\leq s\leq t}$ with Borel state space $\mathcal{X}$ into $2n_t$ parts of length $m_t$, where $t=2n_tm_t$, $n_t\in\N$, $m_t\in\R_+$.
More precisely, for $j\in\{1,\ldots,n_t\}$, define the processes
\[X^{j,1}\coloneqq (X_s)_{s\in [2(j-1)m_t,(2j-1)m_t]},\quad X^{j,2}\coloneqq (X_s)_{s\in [(2j-1)m_t,2jm_t]}.\]
Since $\X$ is a stationary Markov process, the $\beta$-mixing assumption is equivalent to
\[ \Xi(s) \geq \int_{\R^d} \lVert P_s(x,\cdot) - \mu\rVert_{\mathrm{TV}}\, \mu(\d x) = \E\Big[\lVert \Pro(\cdot \vert \mathcal{F}_0) - \Pro\rVert_{\mathrm{TV} \vert \overline{\mathcal{F}}_{s}} \Big] = \E\Big[\lVert \Pro(\cdot \vert \mathcal{F}_t) - \Pro\rVert_{\mathrm{TV} \vert \overline{\mathcal{F}}_{t+s}} \Big],\]
for any $s,t > 0$, see Proposition 1 in \citep{davydov1973}.
Here, $(\mathcal{F}_t = \sigma(X_s, s\leq t))_{t \geq 0}$ denotes the natural filtration of $\X$, $(\overline{\mathcal{F}}_t = \sigma(X_s, s\geq t))_{t \geq 0}$ the filtration of the future of $\X$ and, for a signed measure $\mu$ and a sub-$\sigma$-algebra $\mathcal{A}$ on a measure space $(\Omega, \mathcal{F})$, $\lVert \mu \rVert_{\mathrm{TV} \vert \mathcal{A}}$ denotes the total variation norm of $\mu$ restricted to $\mathcal{A}$.
As demonstrated in \citep[Lemma 1.4]{volk61},
\[\E\Big[\lVert \Pro(\cdot \vert \mathcal{F}_t) - \Pro\rVert_{\mathrm{TV} \vert \overline{\mathcal{F}}_{t+s}} \Big] = \beta(\mathcal{F}_t, \overline{\mathcal{F}}_{t+s}),\]
where for two sub-$\sigma$-algebras $\mathcal{A}, \mathcal{B} \subset \mathcal{G}$ and a probability measure $\PP$ on $(\Omega,\cG)$, the classical $\beta$-mixing coefficient $\beta(\mathcal{A},\mathcal{B})$ is given by
\[\beta(\mathcal{A},\mathcal{B}) = \sup_{C \in \mathcal{A} \otimes \mathcal{B}} \big\lvert\PP\vert_{\mathcal{A} \otimes \mathcal{B}}(C) - \PP\vert_{\mathcal{A}} \otimes \PP\vert_{\mathcal{B}}(C)\big \rvert.\]
Here, $\PP\vert_{\mathcal{A}\otimes \mathcal{B}}$ is the restriction to $(\Omega \times \Omega, \mathcal{A} \otimes \mathcal{B})$ of the image measure of $\PP$ under the canonical injection $\iota(\omega) = (\omega,\omega).$ Clearly, if $\mathcal{A}_1 \subset \mathcal{A}_2$, we have $\beta(\mathcal{A}_1,\mathcal{B}) \leq \beta(\mathcal{A}_2,\mathcal{B}).$
Observe that $X^{j,1}$, as a mapping from $\Omega$ to $\cX^{[2(j-1)m_t, (2j-1)m_t]}$, 
is both $\mathcal{F}_{(2j-1)m_t}$-measurable and $\overline{\mathcal{F}}_{2(j-1)m_t}$-measurable.
It now follows from the above discussion for $j,k \in \{1,\ldots,n_t\}$, $j<k$, that
\begin{align*}
\beta(X^{j,1},X^{k,1}) \coloneqq \beta(\sigma(X^{j,1}),\sigma(X^{k,1})) &\leq \beta(\mathcal{F}_{(2j-1)m_t}, \overline{\mathcal{F}}_{2(k-1)m_t})\\
&\leq  \Xi((2(k-j)-1)m_t)\leq  \Xi((k-j)m_t).
\end{align*}
In the same way, we obtain $\beta(X^{j,2},X^{k,2}) \leq \Xi((k-j)m_t).$
Arguing as in the proof of Proposition 5.1 of \citep{viennet}, we can then construct a process $(\widehat{X}_s)_{0\leq s\leq t}$ by Berbee's coupling method, such that for $k=1,2,$
\begin{enumerate}
\item $X^{j,k}\overset{(\mathrm{d})}{=}\widehat{X}^{j,k}, $ for all $j\in\{1,\ldots,n_t\},$
\item $\Pro(X^{j,k}\neq\widehat{X}^{j,k})\leq \Xi(m_t)$ for all $j\in\{1,\ldots,n_t\},$
\item $\widehat{X}^{1,k},\ldots,\widehat{X}^{n_t,k}$ are independent,
\end{enumerate}
where $\widehat{X}^{j,k}$ is defined analogously to $X^{j,k}$ for $j\in\{1,\ldots,n_t \}$ and $k=1,2.$
In order to ease the notation, define for $j\in\{1,\ldots,n_t\}$ \[I_g(X^{j,1})\coloneqq \int_{2(j-1)m_t}^{(2j-1)m_t}g(X_s)\d s,\quad I_g(X^{j,2})\coloneqq \int_{(2j-1)m_t}^{2jm_t}g(X_s)\d s,\] and, analogously, define $I_g(\widehat{X}^{j,k})$ for $k=1,2,j\in\{1,\ldots,n_t \}$.
Fix $p\geq1$.
Then,
\begin{equation} \label{eq:decomp}
\begin{split}
\left(\E\left[\|\G_t\|_{\mathcal G}^p\right]\right)^{1/p}&\le
\left(\E\left[\sup_{g \in \mathcal{G}}\Big|\frac{1}{\sqrt{t}}\int_0^tg(\hat X_s)\d s\Big|^p\right]\right)^{1/p}+\left(\E\left[\sup_{g \in \mathcal{G}}\Big|\frac{1}{\sqrt{t}}\int_0^t(g(X_s)-g(\hat X_s))\d s\Big|^p\right]\right)^{1/p}\\
&=\left(\E\left[\sup_{g \in \mathcal{G}}\Big|\frac{1}{\sqrt{t}}\sum_{k=1}^2\sum_{j=1}^{n_t}I_g(\hat X^{j,k})\Big|^p\right]\right)^{1/p}\\
&\hspace*{7em}+\left(\E\left[\sup_{g \in \mathcal{G}}\Big|\frac{1}{\sqrt{t}}\sum_{k=1}^2\sum_{j=1}^{n_t}(I_g(X^{j,k})- I_g(\hat X^{j,k}))\Big|^p\right]\right)^{1/p}.
\end{split}
\end{equation}
The classical Bernstein inequality implies for $u>0$ that
\[
\Pro\left(\Big|\frac{1}{\sqrt{t}}\sum_{j=1}^{n_t}I_g(\hat X^{j,k})\Big|>\sqrt{\frac{2n_t\Var\left(\int_0^{m_t}g(X_s)\d s\right)u}{t}}+\frac{m_t\|g\|_\infty u}{\sqrt{t}}\right)\ \le\ \e^{-u},
\]
which in combination with $2n_t \slash t = 1\slash m_t$ yields
\[
\Pro\left(\Big|\frac{1}{\sqrt{t}}\sum_{j=1}^{n_t}I_g(\hat X^{j,k})\Big|> u\right)\ \le\ \exp\left(-\frac{u^2}{2\Big(\mathrm{Var}\big(\tfrac{1}{\sqrt{m_t}}\int_0^{m_t}g(X_s)\d s\big) + \tfrac{m_t}{\sqrt{t}}\lVert g \rVert_\infty u\Big)}\right), \quad u > 0,
\]
see Theorem 2.10 and Corollary 2.11 in \citep{boucheron2013}.
Consequently, denoting
\begin{equation}\label{def:c1c2}
\tilde c_1\coloneqq 2\e^{1/(2\e)}\sqrt{2\pi}\mathrm{e}^{-11/12}, \quad \tilde c_2\coloneqq2(2\e)^{-1/2}\e^{1/(2\e)}\sqrt\pi\e^{1/6},
\end{equation}
Lemma A.2 in \citep{dirk15} gives, for $k\in\{1,2\}$,
\begin{equation}\label{eq:Ig}
\left(\E\left[\Big|\frac{1}{\sqrt{t}}\sum_{j=1}^{n_t}I_g(\hat X^{j,k})\Big|^p\right]\right)^{1/p}\ \le\ \|g\|_\infty \frac{m_t}{\sqrt{t}} \tilde c_1p+\sqrt{\Var\left(\frac{1}{\sqrt{m_t}}\int_0^{m_t}g(X_s)\d s\right)}\tilde c_2\sqrt p,
\end{equation}
where we used again $2n_t \slash t = 1\slash m_t$. In addition, Theorem 3.5 in \citep{dirk15} implies that there exist positive constants $\tilde C_1,\tilde C_2$ such that
\begin{align}\nonumber
\left(\E\left[\sup_{g\in\mathcal G}\Big|\frac{1}{\sqrt{t}}\sum_{j=1}^{n_t}I_g(\hat X^{j,k})\Big|^p\right]\right)^{1/p}
&\le\frac{\tilde C_1}{2}\int_0^\infty\log\mathcal N\big(u,\mathcal G,\tfrac{m_t}{\sqrt{t}}d_\infty\big)\d u+\frac{\tilde C_2}{2}\int_0^\infty\sqrt{\log \mathcal N(u,\mathcal G,d_{\G,m_t})}\d u\\\label{eq:Igsup}
&\hspace*{10em}+2\sup_{g\in\mathcal G}\left(\E\left[\Big|\frac{1}{\sqrt{t}}\sum_{j=1}^{n_t}I_g(\hat X^{j,k})\Big|^p\right]\right)^{1/p}.
\end{align}
Here, we bounded the $\gamma_\alpha$-functionals appearing in the original statement of the theorem by the corresponding entropy integrals.
Note further that the last term on the rhs of \eqref{eq:decomp} is upper bounded by
\begin{align}\nonumber
\left(\E\left[\sup_{g\in\mathcal G}\Big|\frac{1}{\sqrt{t}}\sum_{k=1}^2\sum_{j=1}^{n_t}\left(I_g(X^{j,k})-I_g(\hat X^{j,k})\right)\cdot\1_{X^{j,k}\ne\hat X^{j,k}}\Big|^p\right]\right)^{1/p}
&\le \frac{4 n_t m_t}{\sqrt{t}}\sup_{g\in\mathcal G}\|g\|_\infty\big(\PP\big(X^{j,k}\ne\hat X^{j,k}\big)\big)^{1/p}\\\label{eq:mix}
&\le 2\sup_{g\in\mathcal G}\|g\|_\infty  \sqrt{t}\Xi(m_t)^{1/p}.
\end{align}
Plugging the upper bounds \eqref{eq:Ig}, \eqref{eq:Igsup} and \eqref{eq:mix} into \eqref{eq:decomp} yields
\begin{equation} \label{eq: unimom2}
\begin{split}
\left(\E\left[\sup_{g\in\mathcal G}|\G_t(g)|^p\right]\right)^{1/p}&\le \tilde C_1\int_0^\infty\log\mathcal N\big(u,\mathcal G,\tfrac{m_t}{\sqrt{t}}d_\infty\big)\d u+\tilde C_2\int_0^\infty\sqrt{\log \mathcal N(u,\mathcal G,d_{\G,m_t})}\diff u\\
&\quad+4\sup_{g\in\mathcal G}\Big(\frac{m_t}{\sqrt{t}}\|g\|_\infty \tilde c_1p+\lVert g\rVert_{\mathbb{G},m_t}\tilde c_2\sqrt p + \frac{1}{2}\lVert g \rVert_{\infty}  \sqrt{t} \Xi(m_t)^{1/p}\Big).
\end{split}
\end{equation}
For general $m_t\in(0,\tfrac{t}{4}],$ let $\widetilde{n}_t=\lfloor \frac{t}{2m_t}\rfloor,$ where $\lfloor x \rfloor$ denotes the largest integer smaller or equal to $x\geq1$.
Then, for $\widetilde{m}_t\coloneqq \frac{t}{2\widetilde{n}_t},$ we have $m_t\leq \widetilde{m}_t$, and from $\widetilde{n}_t\geq \frac{t}{2m_t}-1=\frac{t-2m_t}{2m_t}$ and $m_t\leq \frac{t}{4}$, we get
\begin{align*}
\widetilde{m}_t&=\frac{t}{2\widetilde{n}_t}\leq \frac{tm_t}{t-2m_t}\leq 2m_t.
\end{align*}
Since $\widetilde{n}_t\in\N,$ \eqref{eq: unimom2} holds with $\tau = \widetilde{m}_t \in [m_t,2m_t]$ and $m_t$ being replaced by $\tilde{m}_t$, and combining this with the computations above yields
\begin{align*}
\left(\E\left[\sup_{g\in\mathcal G}|\G_t(g)|^p\right]\right)^{1/p}&\le \tilde C_1\int_0^\infty\log\mathcal N\big(u,\mathcal G,\tfrac{\tau}{\sqrt{t}}d_\infty\big)\d u+\tilde C_2\int_0^\infty\sqrt{\log \mathcal N(u,\mathcal G,d_{\G,\tau})}\diff u\\
&\hspace*{3em}+4\sup_{g\in\mathcal G}\Big(\frac{\tau}{\sqrt{t}}\|g\|_\infty \tilde c_1p+\lVert g\rVert_{\mathbb{G},\tau}\tilde c_2\sqrt p + \frac{1}{2}\lVert g \rVert_{\infty}  \sqrt{t} \Xi(\tau)^{1/p}\Big)\\
&\leq \tilde C_1\int_0^\infty\log\mathcal N\big(u,\mathcal G,\tfrac{2m_t}{\sqrt{t}}d_\infty\big)\d u+\tilde C_2\int_0^\infty\sqrt{\log \mathcal N(u,\mathcal G,d_{\G,\tau})}\diff u\\
&\hspace*{3em}+4\sup_{g\in\mathcal G}\Big(\frac{2m_t}{\sqrt{t}}\|g\|_\infty \tilde c_1p+ \lVert g\rVert_{\mathbb{G},\tau}\tilde c_2\sqrt p +\frac{1}{2}\lVert g \rVert_{\infty}  \sqrt{t} \Xi(m_t)^{1/p}\Big),
\end{align*}
which completes the proof.
\end{proof}

\begin{proof}[Proof of Corollary \ref{theo: neumann}]
In case of exponential $\beta$-mixing we obtain, similarly to the proof of Proposition \ref{prop:varmulti}, for any $t > 0$,
\[\lVert g \rVert_{\mathbb{G},t}^2 = \frac{1}{t}\mathrm{Var}\Big(\int_0^t g(X_s)\diff{s} \Big) \leq 2\lVert g \rVert_\infty^2 \int_0^t \int \lVert P_s(x,\cdot) - \mu \rVert_{\mathrm{TV}} \,\mu(\diff{x})\diff{s} \leq 2\lVert g \rVert_\infty^2\frac{c_\kappa}{\kappa}.\]
Choosing $m_T = \sqrt{T}$ and plugging this into \eqref{eq: unimom} therefore yields the assertion for the exponential mixing case. For the $\alpha$-polynomial case we obtain the assertion similarly by the minimizing choice $m_T = T^{p/(\alpha +p)}$, where $T \geq 4^{(\alpha + p)/\alpha}$ guarantees that $m_T \leq T/4$ and the assumption $\alpha > 1$ is needed to guarantee uniform boundedness of $\lVert g \rVert_{\mathbb{G},t}^2$ in $t$.
\end{proof}

For the proof of the bounds on the stochastic error, we start with the following preparatory lemma that provides bounds of the covering numbers of the function class $\mathcal G$ introduced in \eqref{def:G} with respect to the norms appearing in Theorem \ref{Bernstein}. 
By a slight abuse of notation, we do not distinguish notationally between the $\sup$-norm on $\R^d$ and the function space $\mathcal{B}_b(\R^d)$.
\begin{lemma} \label{lemma: covering numbers}
Let $D \subset \R^d$ be a bounded set and, given some Lipschitz continuous kernel $K$ with Lipschitz constant $L$ and compact support $[-1/2,1/2]^d$, define the function class $\mathcal G$ according to \eqref{def:G}.
Then, for any $\varepsilon > 0$,
\[
\mathcal{N}(\varepsilon, \mathcal{G}, \lVert \cdot \rVert_{d_\infty}) \leq \Big(\frac{4L\mathrm{diam}(D)}{\varepsilon h}\Big)^d,\]
and if moreover $\X \in \bm{\Sigma} \cup \bm{\Theta}$, then there exists a constant $\mathbb{A}> 0$ such that, for any $\varepsilon > 0$ and $t > 0$,
\[\mathcal{N}(\varepsilon, \mathcal{G}, \lVert \cdot \rVert_{\mathbb{G},t}) \leq \Big(\frac{2L\mathrm{diam}(D)\sqrt{\mathbb{A} \lVert \rho \rVert_\infty} h^{d-1} \psi_d(h^d)}{\varepsilon}\Big)^d.\]
\end{lemma}
\begin{proof}
For $x \in \R^d$, we obtain by Lipschitz continuity of $K$ that
\begin{equation}\label{eq: covering1}
\begin{split}
B_{d_{\infty}}\big(\overline{K}\big((x- \cdot)/h\big) , \varepsilon \big) &= \big\{\overline{K}\big((y- \cdot)/h\big)  : y \in \R^d, \big\lVert \overline{K}\big((x- \cdot)/h\big)  - \overline{K}\big((y- \cdot)/h\big) \big\rVert_\infty \leq \varepsilon\big\} \\
&\supset \big\{\overline{K}\big((y- \cdot)/h\big)  : y \in \R^d, \lVert x-y \rVert_\infty \leq \varepsilon h \slash (2L)\big\}.
\end{split}
\end{equation}
Let $Q \supset D$ be a cube of side length $\mathrm{diam}(D) < \infty$ and choose for
\[\overline{n} \coloneqq \Big(\Big\lfloor \frac{2L\mathrm{diam}(D)}{\varepsilon h} \Big\rfloor\Big)^d\]
points $x_1,\ldots,x_{\overline{n}} \in Q$ such that $\big\{ B_{d_\infty}(x_i, \varepsilon h/(2L) ): i=1,\ldots \overline{n}\big\}$
covers $Q$ and therefore $D$. 
From \eqref{eq: covering1}, it follows that
$\big\{ B_{d_\infty}\big(\overline{K}\big((x_i- \cdot)/h\big) , \varepsilon \big): i=1,\ldots \overline{n}\big\}$
is an external covering of $\mathcal{G}$.
The external covering number $\mathcal{N}_{\mathrm{ext}}(\varepsilon, \mathcal{G},d_\infty)$ is thus bounded by $(2L\mathrm{diam}(D)\slash(\varepsilon h))^d$. 
Hence,
\[\mathcal{N}(\varepsilon,\mathcal{G},d_\infty) \leq \mathcal{N}_{\mathrm{ext}}(\varepsilon \slash 2, \mathcal{G},d_\infty) \leq \Big(\frac{4L\mathrm{diam}(D)}{\varepsilon h}\Big)^d.\]
Similarly, for 
\begin{equation}\label{def: tilde g}
\tilde{\mathcal{G}} = \{K(x-\cdot)\slash h): x \in D \cap \mathbb{Q}^d\},
\end{equation}
 we obtain 
\[\mathcal{N}(\varepsilon,\tilde{\mathcal{G}},d_\infty) \leq \Big(\frac{2L\mathrm{diam}(D)}{\varepsilon h}\Big)^d.\]
The variance term is bounded by means of Propositions \ref{prop:varextra} and \ref{prop:varmulti}, respectively.
In case $d=1$ for $\X \in \bm{\Sigma}$ or any dimension for $\X \in \bm{\Theta}$, boundedness of $\rho$, Proposition \ref{prop:varextra} and \eqref{eq: mon supp} yield that, for $h \in (0,1)$ and some constant $C$ independent of $\lebesgue(\operatorname{supp}(K((x-\cdot)/h)))=h^d$,
\[\operatorname{Var}\left(\int_0^TK\left(\frac{x-X_t}{h}\right)\d t \right)\leq C(1\vee c_{\tilde{D}}) T\Vert K\Vert_\infty^2\Vert\rho\Vert_\infty h^{2d} \psi^2_d(h^d),\]
where $\tilde{D}$ is a compact set containing $D + [-1/2,1/2]^d$. Hence, for any dimension $d$ and $\X \in \bm{\Sigma} \cup \bm{\Theta}$, we obtain together with Proposition \ref{prop:varmulti} that there exists some global constant $\mathbb{A}$ independent of $h$ such that for any $h \in (0,1)$, $t > 0$ and $g \in \tilde{\mathcal{G}}$,
\begin{equation}\label{eq: main0}
\mathrm{Var}\Big(\frac{1}{\sqrt{t}}\int_0^t g(X_s) \diff{s} \Big) \leq \mathbb{A} \lVert g \rVert_\infty^2 \lVert \rho \rVert_\infty h^{2d} \psi_d^2(h^d),
\end{equation}
and hence 
\begin{equation} \label{eq: main1} 
\lVert g \rVert_{\mathbb{G},t} \leq \sqrt{\mathbb A \lVert \rho \rVert_\infty} h^{d} \psi_d(h^d) \lVert g \rVert_\infty.
\end{equation}
Consequently, with the first part of the proof we obtain
\begin{align*}
\mathcal{N}(\varepsilon, \mathcal{G}, \lVert \cdot \rVert_{\mathbb{G},t}) = \mathcal{N}(\varepsilon, \tilde{\mathcal{G}}, \lVert \cdot \rVert_{\mathbb{G},t}) &\leq \mathcal{N}(\varepsilon (\sqrt{\mathbb{A}\lVert \rho \rVert_\infty}h^d\psi_d(h^d))^{-1},\tilde{\mathcal{G}}, \lVert \cdot \rVert_\infty)\\
&\leq \Big(\frac{2L\mathrm{diam}(D)\sqrt{\mathbb{A} \lVert \rho \rVert_\infty} h^{d-1} \psi_d(h^d)}{\varepsilon}\Big)^d.
\end{align*}
\end{proof}

\begin{proof}[Proof of Lemma \ref{dev:ineq}]
Let $\X \in \bm{\Theta} \cup \bm{\Sigma}$. We start with bounding $\E[\sup_{x\in D}\vert\mathbb H_{h,T}(x)\vert^p]$. 
Let $m_T \in (0,T\slash 4]$ and $\tau \in [m_T,2m_T]$ as in Theorem \ref{Bernstein}.
Using \eqref{eq: main1} and $\sup_{f,g\in\tilde{\mathcal G}}\|f- g\|_\infty\le 2\|K\|_\infty$ for $\tilde{\mathcal{G}}$ defined in \eqref{def: tilde g}, we obtain
\begin{equation}\label{eq: mathbb V}
\sup_{f,g\in \tilde{\mathcal G}}\|f-g\|_{\G,\tau}\le \sqrt{\mathbb{A} \lVert \rho \rVert_\infty} \sup_{f,g\in\tilde{\mathcal G}}\|f-g\|_\infty h^d\psi_d(h^d)\le 2\sqrt{\mathbb{A}\lVert \rho \rVert_\infty}\|K\|_\infty h^d\psi_d(h^d)\eqqcolon \mathbb V(h),
\end{equation}
such that $\mathcal N(u,\tilde{\mathcal G},\|\cdot\|_{\G,\tau})=1$ for $u\ge \mathbb V(h)$.
Consequently, using the bound $\int_0^C \sqrt{\log(M\slash u)} \diff{u} \leq 4C \sqrt{\log(M\slash C)}$ provided $\log(M\slash C) \geq 2$, see e.g.\ p.\ 592 of Gin\'e and Nickl \citep{gini09}, and the covering number bound from Lemma \ref{lemma: covering numbers}, it follows for $h \leq \mathrm{e}^{-2}L\mathrm{diam}(D) \slash \lVert K \rVert_\infty$ that
\begin{align*}
\int_0^\infty\sqrt{\log\mathcal N(u,\mathcal G,d_{\G,\tau})} \d u = \int_0^\infty\sqrt{\log\mathcal N(u,\tilde{\mathcal G},d_{\G,\tau})} \d u&\le\int_0^{\mathbb V(h)}\sqrt{d \ \log\left(\frac{L\mathrm{diam}(D)\mathbb V(h)}{u h\lVert K\rVert_\infty}\right)}\d u\\
&\le 2\mathbb{V}(h)\sqrt{d\log\left(\frac{L \mathrm{diam}(D)}{\|K\|_\infty h}\right)}.
\end{align*}
Moreover, since $\sup_{f,g \in \mathcal{G}} \lVert f-g\rVert_\infty \leq 4\lVert K\rVert_\infty$, it follows that $\mathcal{N}(u,\mathcal{G},d_\infty) = 1$ for all $u \geq 4 \lVert K \rVert_\infty$ and hence we obtain by the covering number bound with respect to the $\sup$-norm from Lemma \ref{lemma: covering numbers} and elementary calculations
\[\int_0^\infty\log \mathcal N\big(u,\mathcal G,\tfrac{2m_T}{\sqrt{T}}d_\infty\big)\d u = 2\tfrac{m_T}{\sqrt{T}} \int_0^{4 \lVert K \rVert_\infty}\log \mathcal N(u,\mathcal G,d_\infty)\d u \leq 8\tfrac{m_T}{\sqrt{T}} d \lVert K \rVert_\infty \Big(1 + \log \Big(\frac{L \mathrm{diam}(D)}{\Vert K \rVert_\infty h} \Big) \Big).\]
Denseness of $\mathbb{Q}^d$ in $\R^d$, continuity of $x \mapsto \mathbb{H}_{h,T}(x)$ and Theorem \ref{Bernstein} thus imply for $h \leq \mathrm{e}^{-2}L\mathrm{diam}(D) \slash \lVert K \rVert_\infty$
\begin{equation}
\begin{split}\label{ineq:Hp}
&\left(\E\Big[\sup_{x\in D}|\mathbb{H}_{h,T}(x)|^p\Big]\right)^{1/p}=\left(\E\Big[\sup_{x\in D\cap \mathbb{Q}^d}|\mathbb{H}_{h,T}(x)|^p\Big]\right)^{1/p}\\
&\qquad\le \frac{1}{\sqrt Th^d}\bigg(8\tilde C_1 \frac{m_T}{\sqrt{T}} d \lVert K \rVert_\infty \left(1 + \log \left(\frac{L \mathrm{diam}(D)}{\Vert K \rVert_\infty h} \right) \right)+2\tilde C_2 \mathbb{V}(h)\sqrt{d\log \left(\frac{L \mathrm{diam}(D)}{\Vert K \rVert_\infty h} \right)}\\
&\qquad\hspace*{3em}+16\frac{m_T}{\sqrt{T}}\|K\|_\infty\tilde c_1 p+2\mathbb{V}(h)\tilde c_2\sqrt p + 4\lVert K\rVert_\infty \sqrt{T} \Xi(m_T)^{1/p}\bigg),
\end{split}
\end{equation}
for $\mathbb V(h)$ introduced in \eqref{eq: mathbb V}.
Now, let $p = u_T \geq 1$ be such that $\Xi^{-1}(T^{-u_T}) \in \mathsf{o}(T)$. Then, for $T$ large enough, \eqref{ineq:Hp}, $h \geq T^{-2}$ and $h \in \mathsf{o}(1)$ imply for the choice $m_T = \Xi^{-1}(T^{-u_T})$ that
\begin{align*}
&\E\big[\big\|\hat\rho_{h,T}-\E\hat\rho_{h,T}\big\|_{L^\infty(D)}^{u_T}\big]
\\
&\quad \leq c^{u_T}\bigg(\frac{\log T}{Th^d} \Xi^{-1}(T^{-u_T}) +T^{-\frac 1 2}\psi_d(h^d)\sqrt{\log(h^{-1})}+\frac{u_T}{Th^d} \Xi^{-1}(T^{-u_T}) +T^{-\frac 1 2} \psi_d(h^d)\sqrt{u_T} + h^{-d}T^{-1}\bigg)^{u_T}
\\
&\quad \leq c^{u_T}\bigg(\frac{\log T + u_T}{Th^d} \Xi^{-1}(T^{-u_T}) +T^{-\frac 1 2}\psi_d(h^d)\big(\sqrt{\log (h^{-1})} + \sqrt{u_T}\big)\bigg)^{u_T},
\end{align*}
where the value of the constant $c$ changes from line to line. Hence Markov's inequality implies that there exists some constant $c^\ast > 0$ such that 
\begin{equation} \label{eq: dev1}
\begin{split}
\Pro\left(\big\|\hat\rho_{h,T}-\E\hat\rho_{h,T}\big\|_{L^\infty(D)}\geq c^\ast\left( \frac{u_T + \log T}{Th^d} \Xi^{-1}(T^{-u_T}) +T^{-\frac 1 2}\psi_d(h^d)\sqrt{u_T \vee \log(h^{-1})}\right) \right)\leq \e^{-u_T}.
\end{split}
\end{equation}
Suppose now that $\X \in \bm{\Sigma}$. Then, $\X$ is exponentially $\beta$-mixing, i.e., $\Xi(t) = c_\kappa \mathrm{e}^{-\kappa t}$, where without loss of generality we may assume that $c_\kappa \geq 1$. Then, for any $\gamma > 0$ and $1 \leq u_T \leq \gamma \log T$, it follows from $\Xi^{-1}(T^{-u_T}) \leq u_T \log T\slash \kappa$ and \eqref{eq: dev1} that there exists some constant $c_\gamma > 0$ such that 
\begin{align*}
\Pro\left(\big\|\hat\rho_{h,T}-\E\hat\rho_{h,T}\big\|_{L^\infty(D)}\geq c_\gamma\left( \frac{u_T(\log T)^2}{Th^d} +T^{-\frac 1 2}\psi_d(h^d)\sqrt{u_T \vee \log(h^{-1})}\right) \right)\leq \e^{-u_T}.
\end{align*}
\end{proof}

\section{Proofs for Section \ref{sec: density estimation}}\label{app: density estimation}
\begin{proof}[Proof of Corollary \ref{cor:MSE}]
Fix $x$ such that there exists an open neighbourhood $D$ of $x$ such that $\rho\vert_D \in \mathcal{H}_D(\beta,\mathsf{L})$.
The usual bias-variance decomposition gives
\begin{equation}\label{stat:pytha}
\E\left[\left(\hat\rho_{h,T}(x)-\rho(x)\right)^2\right]=\left(\rho\ast K_h(x)-\rho(x)\right)^2+\Var\left(\hat\rho_{h,T}(x)\right).
\end{equation}
For the bias term, since $\llfloor \beta \rrfloor \leq \ell$, there exists a universal constant $M>0$ such that
\begin{equation} \label{eq: rate1}
\vert(\rho\ast K_h-\rho)(x)\vert=\left|h^{-d}\int K\left(\frac{x-y}{h}\right)(\rho(y)-\rho(x))\d y\right| \le Mh^\beta,
\end{equation}
see Proposition 1.2 in \cite{tsy09} for the case $d=1$ and the analogous estimator for discrete observations, which can be extended to the general multivariate case under continuous observations without much effort.
Moreover, for any dimension $d$ and $\X \in \bm{\Sigma} \cup \bm{\Theta}$, it follows from \eqref{eq: main0} that for any $h \in (0,1)$
\[
\operatorname{Var}\left(\frac{1}{T}\int_0^TK_h(x-X_t)\d t \right)\lesssim T^{-1} \Vert K\Vert_\infty^2\Vert\rho\Vert_\infty \psi_d^2(h^d).
\]
The claim follows by plugging the specific choice of $h$ into \eqref{eq: rate1} and \eqref{eq: main0} and using \eqref{stat:pytha}.
\end{proof}

\begin{proof}[Proof of Theorem \ref{theo:invdens}]
Fix $p\geq1$, and recall the decomposition \eqref{eq: rate0}.
By the assumption on the order of the kernel $K$, the bias term $\rho\ast K_h-\rho$ is bounded by $B(h) \coloneqq Mh^\beta$ for some universal constant $M > 0$ as in the pointwise case (see \eqref{eq: rate1}), while the upper bound on the stochastic error $\mathbb{H}_{h,T}$ relies on a suitable specification on the upper bound in \eqref{ineq:Hp}.
For $d\geq 3$, set $h = h(T) = (\log T\slash T)^{1\slash (2\beta +d -2)}$ and $m_T = p\log T\slash \kappa$ such that
\[\frac{1}{\sqrt{T}} \psi_d(h^d) \in \mathsf{O}\big(T^{-\beta\slash(2\beta + d -2)}\big)\quad\text{ and }\quad
\frac{m_T}{Th^d} = \Big( \tfrac{\log T}{T} \Big)^{\frac{2(\beta-1)}{2(\beta-1)+d}}.\]
Upon noting that $\beta > 2$ implies $2(\beta -1) > \beta$, it follows from \eqref{ineq:Hp} that
\begin{equation}\label{eq: rate3}
\left(\E\Big[\sup_{x\in D}|\mathbb{H}_{h,T}(x)|^p\Big]\right)^{1/p} \in \mathsf{O}\Big(\Big(\tfrac{\log T}{T}\Big)^{\beta\slash(2\beta + d -2)}\Big).
\end{equation}
Since $h^\beta = (\log T \slash T)^{\beta\slash(2\beta + d -2)}$, \eqref{eq: rate0}, \eqref{eq: rate1} and \eqref{eq: rate3} finally give
$\E[\lVert \hat{\rho}_{h,T}-\rho \rVert_{L^\infty(D)}^p]^{1/p}  \in \mathsf{O}\big(\Psi_{d,\beta}(T)\big)$ for $d \geq 3$.
For $d=1$ and $d=2$, the assertion follows by analogous arguments. 

We now proceed with the proof of the convergence rate of the adaptive scheme for $d \geq 3$. For the variance, we obtain from \eqref{ineq:Hp} that, for $m_T\coloneqq 2\log_{(k)} T(\log T)^2/\kappa$ and whenever $h\le \mathrm{e}^{-2}L\mathrm{diam}(D)/\|K\|_\infty$, there exists some constant $\const>0$ such that
\[
\E\left[\left\|\hat\rho_{h,T}-\E \hat\rho_{h,T}\right\|_{L^\infty(D)}^2\right]=
\E\left[\sup_{x\in D}|\mathbb H_{h,T}(x)|^2\right]\le \const^2 \sigma^2(h,T),\]
where $\sigma^2(\cdot,\cdot)$ is defined according to \eqref{def:sigma}.
Define $h_{\rho}$ by the balance equation
\[h_{\rho}\coloneqq \max\bigg\{h\in\H_T: B(h)\le\frac14\sqrt{0.8\M}\sigma(h,T)\bigg\},\quad\text{ where }\M\coloneqq \|\rho\|_{L^\infty(D)}.
\]
This definition implies that $B(h_\rho)\simeq \sqrt{0.8\M}\sigma(h_\rho,T)/4$ and, since $\H_T\ni h_\rho>\left(\frac{\log_{(k)} T(\log T)^{5}}{T}\right)^{\frac{1}{d+2}}$,
\[
h_\rho^{2\beta+d-2}\simeq \frac{\log_{(k)} T\log T}{T}\quad \text{ and }\quad \sigma(h_\rho,T)\simeq \left(\frac{\log_{(k)} T\log T}{T}\right)^{\frac{\beta}{2\beta+d-2}}.
\]
To justify this, define $h_0\coloneqq(\log_{(k)}T\log T/T)^{1/(2\beta+d-2)}$. 
For large enough $T$, we have the bound $\log(\log_{(k)} T \log T) \leq (\log T)/2$ and hence
\begin{align*}
\sigma(h_0,T)&=\frac{\log_{(k)}T(\log T)^2}{Th_0^d}\log(h_0^{-1})+\psi_d(h_0^d)\sqrt{\frac{\log_{(k)}T\log(h_0^{-1})}{T}}\\
&\geq \sqrt{\frac{\log_{(k)}T\log T}{2(2\beta +d -2 )T}}\psi_d(h_0^d)=\sqrt{\frac{1}{2(2\beta +d -2 )}}\left(\frac{\log_{(k)} T\log T}{T}\right)^{\frac{\beta}{2\beta+d-2}}=\mathcal L^{-1} B(h_0),
\end{align*}
for $\mathcal L=\sqrt{2(2\beta + d -2)M^2}$.
Additionally, we get, since $\beta>2$,
\[
\sigma(h_0,T)=\frac{\log_{(k)}T(\log T)^2}{Th_0^d}\log(h_0^{-1})+\psi_d(h_0^d)\sqrt{\frac{\log_{(k)}T\log(h_0^{-1})}{T}}\simeq \left(\frac{\log_{(k)}T\log T}{T}\right)^{\frac{\beta}{2\beta+d-2}}.
\]
In particular, it holds that $h_0\lesssim h_\rho$, which
is clear if $\mathcal L\le\frac14\sqrt{0.8\M}$, and else  follows by the fact that, for any $0<\lambda<1$,
\[
	B(\lambda h_0)=\lambda^\beta B(h_0)\leq \lambda^\beta\mathcal L\sigma(h_0,T)\leq \lambda^\beta\mathcal L\sigma(\lambda h_0,T).
\]
Lastly, we show $h_\rho\lesssim h_0$ by proving $h_\rho^{2\beta+d-2}h_0^{-(2\beta+d-2)}\in\mathsf{O}(1).$ 
Indeed, by the definition of $h_\rho$,
\begin{align*}
h_\rho^{2\beta+d-2}&\lesssim h_\rho^{d-2}\sigma^2(h_\rho,T)
\\
&\lesssim h_\rho^{d-2}\left(\frac{\log_{(k)}T(\log T)^3}{T}h_\rho^{-d}+\psi_d(h_\rho^d)\sqrt{\frac{\log_{(k)}T\log T}{T}}\right)^2\\
&\lesssim \frac{(\log_{(k)}T)^2(\log T)^6}{T^2}h_\rho^{-(2+d)}+h_\rho^{d-2}\psi_d^2(h_\rho^d)\frac{\log_{(k)}T\log T}{T},
\end{align*}
and thus it holds that
\[
h_\rho^{2\beta+d-2} h_0^{-(2\beta+d-2)}\lesssim\frac{\log_{(k)}T(\log T)^5}{T}h_\rho^{-(2+d)}+h_\rho^{d-2}\psi_d^2(h_\rho^d)\in\mathsf{O}(1),
\]
thanks to $h_\rho > (\log_{(k)}T (\log T)^5/T)^{1/(d+2)}$.
\paragraph{Case 1:}
We first consider the case where $\hat h_T\ge h_{\rho}$.
To shorten notation, denote $\tilde{\M}\coloneqq \|\hat\rho_{h_{\min},T}\|_{L^\infty(D)}$.
Then, exploiting the definition of $\hat h_T$ according to \eqref{est:band0} and the bias and variance bounds,
\begin{align*}
&\E\left[\|\hat\rho_{\hat h_T,T}-\rho\|_{L^\infty(D)}\cdot\mathbf{1}_{\{\hat h_T\ge h_{\rho}\}\cap \{\tilde{\M}\le 1.2\M \}}\right]\\
&\hspace*{3em}\le \E\left[\left(\|\hat\rho_{\hat h_T,T}-\hat\rho_{h_{\rho},T}\|_{L^\infty(D)}+\|\hat\rho_{h_{\rho},T}-\E\hat\rho_{h_{\rho},T}\|_{L^\infty(D)}+B(h_\rho)\right)\mathbf{1}_{\{\hat h_T\ge h_{\rho}\}\cap \{\tilde{\M}\le 1.2\M \}}\right]\\
&\hspace*{3em}\le \sqrt{1.2\M}\sigma(h_{\rho},T)+\const\sigma(h_{\rho},T)+\frac14\sqrt{0.8\M}\sigma(h_{\rho},T) \in \mathsf{O}(\sigma(h_{\rho},T)).
\end{align*}
Similarly,
\begin{align*}
&\E\left[\|\hat\rho_{\hat h_T,T}-\rho\|_{L^\infty(D)}\cdot\mathbf{1}_{\{\hat h_T\ge h_{\rho}\}\cap \{\tilde{\M}>1.2\M \}}\right]\\
&\hspace*{3em}\le \sum_{h\in\H_T: h\ge h_{\rho}}\E\left[\left(\|\hat\rho_{h,T}-\E\hat\rho_{h,T}\|_{L^\infty(D)}+B(h)\right)\cdot\mathbf{1}_{\{\hat h_T=h\}\cap \{\tilde{\M}>1.2\M \}}\right]\\
&\hspace*{3em}\lesssim \log T\left(\const\sigma(h_{\rho},T)+B(1)\right)\sqrt{\PP(\tilde{\M}>1.2\M)}.
\end{align*}
Now, for any $T$ large enough,
\begin{equation}\begin{split}\label{ineq:M}
\Pro\left(\big|\tilde{\M}-\M\big|>0.2\|\rho\|_{L^\infty(D)}\right)&=\Pro\left(\big|\|\hat\rho_{h_{\min},T}\|_{L^\infty(D)}-\|\rho\|_{L^\infty(D)}\big|>0.2\M\right)\\
&\le \Pro\left(\big\|\hat\rho_{h_{\min},T}-\rho\big\|_{L^\infty(D)}>0.2\|\rho\|_{L^\infty(D)}\right)\\
&\le \Pro\left(\big\|\hat\rho_{h_{\min},T}-\E\hat\rho_{h_{\min},T}\big\|_{L^\infty(D)}>0.2\|\rho\|_{L^\infty(D)}-B(h_{\min})\right)\\
&\le \Pro\left(\big\|\hat\rho_{h_{\min},T}-\E\hat\rho_{h_{\min},T}\big\|_{L^\infty(D)}>0.1\|\rho\|_{L^\infty(D)}\right)\\
&\le \Pro\left(\big\|\hat\rho_{h_{\min},T}-\E\hat\rho_{h_{\min},T}\big\|_{L^\infty(D)}>\Upsilon_{h_{\min},T}(\log T)\right)\\
&\le T^{-1},
\end{split}\end{equation}
where, for the function $\Upsilon_{h_{\min},T}(\cdot)$ defined according to \eqref{def:Ups}, the last inequality follows from Lemma \ref{dev:ineq} and the last but one inequality holds since there exists some constant $C$ such that
\begin{align*}
\Upsilon_{h_{\min},T}(\log T)&\le CT^{-\frac{2}{d+2}}\left((\log T)^{\frac{6-2d}{d+2}}(\log_{(k)}T)^{-\frac{d}{d+2}}+(\log T)^{\frac{6-3d}{d+2}}(\log_{(k)}T)^{\frac{2-d}{2(d+2)}}\right)\\
&\le 0.2\|\rho\|_{L^\infty(D)},
\end{align*}
for $T$ sufficiently large.
Thus, we can conclude that $\E\left[\|\hat\rho_{\hat h_T,T}-\rho\|_{L^\infty(D)}\cdot\mathbf{1}_{\{\hat h_T\ge h_{\rho}\}}\right] \in \mathsf{O}(\sigma(h_{\rho},T))$.

\paragraph{Case 2:}
For the case $\hat h_T<h_{\rho}$, note first that the previous bias and variance bounds together with \eqref{ineq:M} imply that
\begin{align*}
&\E\left[\|\hat\rho_{\hat h_T,T}-\rho\|_{L^\infty(D)}\cdot\mathbf{1}_{\{\hat h_T< h_{\rho}\}\cap \{\tilde{\M}<0.8\M\}}\right]\\
&\hspace*{3em}\le \sum_{h\in\H_T: h<h_{\rho}}\E\left[\left(\|\hat\rho_{h,T}-\E\hat\rho_{h,T}\|_{L^\infty(D)}+B(h)\right)\cdot\mathbf{1}_{\{\hat h_T=h\}\cap\{\tilde{\M}<0.8\M\}}\right]\\
&\hspace*{3em}\lesssim \log T\left(\const\sigma(h_{\min},T)+B(h_{\rho})\right)\sqrt{\PP(\tilde{\M}<0.8\M)} = O(\sigma(h_{\rho},T)).
\end{align*}
On the other hand,
\begin{align*}
&\E\left[\|\hat\rho_{\hat h_T,T}-\rho\|_{L^\infty(D)}\cdot\mathbf{1}_{\{\hat h_T<h_{\rho}\}\cap \{0.8\M\le\tilde{\M}\}}\right]\\
&\hspace*{3em}\le \sum_{h\in\H_T: h<h_{\rho}}
\E\left[\left(\|\hat\rho_{h,T}-\E\hat\rho_{h,T}\|_{L^\infty(D)}+B(h)\right)\cdot\mathbf{1}_{\{\hat h_T=h\}\cap\{0.8\M\le\tilde{\M}\}}\right]\\
&\hspace*{3em}\le \sum_{h\in\H_T: h<h_{\rho}}
\sqrt{\E\left[\|\hat\rho_{h,T}-\E\hat\rho_{h,T}\|_{L^\infty(D)}^2\right]}\sqrt{\E\left[\mathbf{1}_{\{\hat h_T\ge h_{\rho}\}\cap\{0.8\M\le \tilde{\M}\}}\right]}+B(h_{\rho})\\
&\hspace*{3em}\le \sum_{h\in\H_T: h<h_{\rho}}
\const\sigma(h,T)\sqrt{\P\left(\{\hat h_T\ge h_{\rho}\}\cap\{0.8\M\le\tilde{\M}\}\right)}+\mathsf{O}(\sigma(h_{\rho},T)).
\end{align*}
Pick any $h\in\H_T$ such that $h<h_\rho$ and denote $h^+\coloneqq \min\{g\in\H_T:g>h\} = \eta h$.  
It is then shown as in the proof of Theorem 2 in \cite{gini09} that the verification of the fact that the first sum on the rhs of the last display is of order $\mathsf{O}(\sigma(h_{\rho},T))$ boils down to proving that
\[
\sum_{h\in\H_T: h<h_\rho}\sigma(h,T)\left(\sum_{g\in\H_T: g\le h}\P\left(\big\|\hat\rho_{h^+,T}-\hat\rho_{g,T}\big\|_{L^\infty(D)}>\sqrt{0.8\M}\sigma(g,T)\right)\right)^{1/2}
\in \mathsf{O}(\sigma(h_\rho,T)).
\]
Following again the lines of \cite{gini09}, we obtain
\begin{align*}
\Pro\left(\big\|\hat\rho_{h^+,T}-\hat\rho_{g,T}\big\|_{L^\infty(D)}>\sqrt{0.8\M}\sigma(g,T)\right)&\le
\Pro\left(\big\|\hat\rho_{h^+,T}-\E\hat\rho_{h^+,T}\big\|_{L^\infty(D)}>\frac14\sqrt{0.8\M}\sigma(h^+,T)\right)\\
&\quad+ \Pro\left(\big\|\hat\rho_{g,T}-\E\hat\rho_{g,T}\big\|_{L^\infty(D)}>\frac14\sqrt{0.8\M}\sigma(g,T)\right).
\end{align*}
Let $\gamma \geq 1$.
Clearly, by definition of $\sigma(g,T)$, there exists $T(\gamma) > 0$ such that, for any $T \geq  T(\gamma)$ and any $g \leq h_\rho$, $g \in \H_T$, we have 
\[\frac{1}{4}\sqrt{0.8 \mathcal{M}}\sigma(g,T) \geq c_\gamma \Upsilon_{g,T}(\gamma\log(g^{-1}))=c_\gamma \frac{\gamma\log(g^{-1})(\log T)^2}{Tg^d}+\psi_d(g^d)\sqrt{\frac{\gamma\log(g^{-1})}{T}},\]
where $c_\gamma$ is the constant appearing in Lemma \ref{dev:ineq}. 
Thus, using Lemma \ref{dev:ineq}, we obtain for $T \geq T(\gamma)$ that 
\[\Pro\left(\big\|\hat\rho_{g,T}-\E[\hat\rho_{g,T}]\big\|_{L^\infty(D)}>\frac14\sqrt{0.8\M}\sigma(g,T)\right) \leq \mathrm{e}^{-\gamma \log(g^{-1})} = g^\gamma \eqqcolon \iota_\gamma(g)\]
and hence
\[
\sum_{g\in\H_T:g\le h}\P\left(\big\|\hat\rho_{h^+,T}-\hat\rho_{g,T}\big\|_{L^\infty(D)}>\sqrt{0.8\M}\sigma(g,T)\right)\le \sum_{g\in\H_T:g\le h}(\iota_\gamma(g)+\iota_\gamma(h^+))
\le 2\iota_\gamma(h)\log T.
\]
Thus, choosing $\gamma$ large enough demonstrates that
\begin{align*}
&\sum_{h\in\H_T: h<h_\rho}\sigma(h,T)\left(\sum_{g\in\H_T: g\le h}\P\left(\big\|\hat\rho_{h^+,T}-\hat\rho_{g,T}\big\|_{L^\infty(D)}>\sqrt{0.8\M}\sigma(g,T)\right)\right)^{1/2}\\
&\hspace*{3em} \leq\sum_{h\in\H_T: h<h_\rho}\sigma(h,T)\sqrt{2\iota_\gamma(h)\log T}\leq \sqrt{2h_\rho^\gamma(\log T)^3} \sigma(h_{\min},T) \in \mathsf{O}(\sigma(h_\rho,T)),
\end{align*}
as desired.
\end{proof}

\begin{proof}[Proof of Theorem \ref{theo: ou}]
Let us first verify that under \ref{ou2} the heat kernel bound \ref{ass: density bound} holds. Arguing as in the proof of Theorem 3.2 of Masuda \citep{masuda2004}, we see that $\mathscr{F}\mu$ and $\varphi_{X_t}^x$ are integrable for any $x \in \R$ and $t > 0$ and hence we can obtain the invariant density $\rho$ and the transition  density $p_t$ of $\X$ via inverse Fourier transformation through
\[\rho(y) = \frac{1}{(2\pi)^d} \int_{\R^d} \mathrm{e}^{-\mathrm{i}\langle y, \lambda \rangle} \{ \mathscr{F}\mu \}(\lambda)\diff{\lambda}, \quad y \in \R^d,\]
and
\[p_t(x,y) = \frac{1}{(2\pi)^d} \int_{\R} \mathrm{e}^{-\mathrm{i}\langle y, \lambda \rangle} \varphi_{X_t}^x(\lambda)\diff{\lambda}, \quad x,y \in \R^d, t > 0.\]
Again, as in the proof of Theorem 3.2 in \cite{masuda2004}, it follows that under \ref{ou2}, 
\[
\big \lvert \varphi_{X_t}^x(\lambda) \big\rvert \leq \exp\Big(-\frac{1}{2}\lambda^\top \Big(\int_0^t \mathrm{e}^{-sB} Q \mathrm{e}^{-sB^\top} \diff{s} \Big) \lambda \Big), \quad x,\lambda \in \R^d, t > 0.
\]
Thus, using the characterization of the multivariate normal distribution, we obtain 
\begin{align*} 
p_t(x,y) &\leq \frac{1}{(2\pi)^d} \int_{\R^d} \exp\Big(-\frac{1}{2}\lambda^\top \Big(\int_0^t \mathrm{e}^{-sB} Q \mathrm{e}^{-sB^\top} \diff{s} \Big) \lambda \Big) \diff{\lambda}\\
&= \frac{1}{(2\pi)^{d/2}} \Big(\det\Big(\int_0^t \mathrm{e}^{-sB} Q \mathrm{e}^{-sB^\top} \diff{s} \Big) \Big)^{-1/2}.
\end{align*}
Observing that 
\begin{align*}
\lim_{t \downarrow 0} t^{d/2} \Big(\det\Big(\int_0^t \mathrm{e}^{-sB} Q \mathrm{e}^{-sB^\top} \diff{s} \Big) \Big)^{-1/2} &= \Big(\det\Big(\lim_{t \downarrow 0}\frac{1}{t}\int_0^t \mathrm{e}^{-sB} Q \mathrm{e}^{-sB^\top} \diff{s} \Big) \Big)^{-1/2}\\
&= \det(Q)^{-1/2} < \infty,
\end{align*}
where finiteness is a consequence of invertibility of $Q$ by \ref{ou2}, it follows that for any $d \geq 1$, there exists a constant $c > 0$ such that 
\[\sup_{x,y \in \R^d} p_t(x,y) \leq ct^{-d/2}, \quad t \in (0,1].\]
Thus indeed, for any dimension $d \in \N$, \ref{ass: density bound} holds. 
Next, in scenario \ref{prop: ou2}, \cite[Theorem 4.3]{masuda2004} gives the exponential $\beta$-mixing property and the proof of Theorem 2.6 in \cite{MASUDA200735} along with \cite[Proposition 3.8]{MASUDA200735} yields $V$-exponential ergodicity with $V(x) \sim (1 +\lVert x \rVert^p)$. This together with \eqref{eq: ou heat} entails that in scenario \ref{prop: ou2}, we have $\mathbf{X} \in \bm{\Sigma} \cap \bm{\Theta}$. Finally, $\X \in \bm{\Theta}$ in scenario \ref{prop: ou3} follows from the considerations above and Lemma \ref{prop: poisson} due to the fact that the combination of \ref{ou2} and the logarithmic moment condition imply that every compact set is small and hence petite since $\X$ is strong Feller and by \cite[Theorem 3.1]{jongbloed2005} ergodic (see Proposition \ref{prop2:betamix}) and hence \ref{ou7} implies $V$-polynomial ergodicity of degree $\alpha -1 > 1$ with $V(x) = C(\log \lvert x \rvert)^\alpha$ in dimension $d=1$ by \cite[Corollary 1]{kevei2018}. The statements on the estimation rates are now an immediate consequence of Corollary \ref{cor:MSE} and Theorem \ref{theo:invdens} and the fact that $\rho \in \mathcal{C}^\infty_b$ has arbitrary H\"older smoothness.
\end{proof}

\begin{proof}[Proof of Lemma \ref{Chen and Applebaum assumptions}]
We will employ Theorem 6.2.9 and Exercise 6.4.7 of \citep{applebaum09} to show the first assertion. So first we must verify that  condition \textbf{(C1)} on page 365 of \citep{applebaum09} holds. Since \ref{Lipschitz assumptions} holds, we only have to show that there exists a constant $K_1>0$ such that, for all $x,y\in\R^d$,
\begin{equation*}
\sum_{i,j=1}^{d}(\sigma_{i,j}(x)-\sigma_{i,j}(y))^2+\int_{\R^d}\Vert \gamma(x)z-\gamma(y)z\Vert^2\,\nu(\mathrm{d}z)\leq K_1\Vert x-y\Vert^2 \label{c1 applebaum},
\end{equation*}
where $\sigma_{i,j}(x)$ denotes the components of $\sigma(x)\in\R^{d\times d}$ for any $x\in\R^d$. \ref{Lipschitz assumptions} implies that there exists a finite constant $L_{i,j}>0$ for any $i,j\in\{1,\ldots,d\},$ such that $\sigma_{i,j}\colon \R^d\to \R$ is Lipschitz continuous with Lipschitz constant $L_{i,j}>0$ and hence we have for $x,y\in\R^d$
\begin{align*}
\sum_{i,j=1}^d(\sigma_{i,j}(x)-\sigma_{i,j}(y))^2
&\leq 2d \max_{i,j\in\{1,\ldots,d\}}L_{i,j}^2\Vert x-y\Vert^2.
\end{align*}
Furthermore, we have for $x,y\in\R^d$ by the Lipschitz continuity of $\gamma$
\begin{align*}
\int_{\R^d}\Vert \gamma(x)z-\gamma(y)z\Vert^2\,\nu(\mathrm{d}z)
&\leq L_\gamma^2\Vert x-y\Vert^2\int_{\R^d}\Vert z\Vert^2\,\nu(\mathrm{d}z),
\end{align*}
where we denote the Lipschitz constant of $\gamma$ by $L_\gamma$. 
By \ref{Ergodicity Assumptions}, $\int_{\R^d}\Vert z\Vert^2\,\nu(\mathrm{d}z)$ is finite and hence \textbf{(C1)} holds.
To verify the growth condition \textbf{(C2)} on page 366 of \citep{applebaum09}, we have to show that there exists a constant $K_2$ such that, for all $x\in\R^d$, \[\int_{\R^d} \Vert \gamma(x)z\Vert^2\,\nu(\mathrm{d}z)\leq K_2(1+\Vert x\Vert ^2).\] Since $\gamma$ is Lipschitz continuous by \ref{Lipschitz assumptions}, there exists a constant $K>0$ such that the linear growth condition $\Vert \gamma(x)\Vert \leq K(1+\Vert x\Vert)$ holds for all $x\in\R^d$, and thus we have, for $x\in\R^d$,
\begin{align*}
\int_{\R^d} \Vert \gamma(x)z\Vert^2\,\nu(\mathrm{d}z)
&\leq 2K^2(1+\Vert x\Vert^2)\int_{\R^d}\Vert z\Vert^2\,\nu(\mathrm{d}z).
\end{align*}
Again by \ref{Ergodicity Assumptions}, $\int_{\R^d}\Vert z\Vert^2\,\nu(\mathrm{d}z)$ is finite and hence \textbf{(C2)} holds for $K_2=2K^2\int_{\R^d}\Vert z\Vert^2\,\nu(\mathrm{d}z)$. Since Assumption 6.2.8 in \citep{applebaum09} is trivially fulfilled, the first assertion follows by Theorem 6.2.9 and Exercise 6.4.7 of \citep{applebaum09}.\\ We proceed by showing the second assertion. Equation (1.21) of \citep{CHEN20176576} is in the setting of \eqref{SDE}  equivalent to $\kappa_\alpha(x,z)=\Vert \gamma(x)z\Vert^{d+\alpha}\nu(z)\geq0$ for all $x\in\R^d$ and almost every $z\in\R^d$. Since $\nu$ is a density, this assumption is fulfilled.\\ 
For assumption $(\textbf{H}^a)$ of \citep{CHEN20176576} to hold, we only need to show that there exists a $\beta\in(0,1)$ such that the function $a(x)\coloneqq \sigma(x)\sigma^\top(x)$ is $\beta$-H\"older continuous. However this follows directly from the Lipschitz continuity and the boundedness of $\sigma$ imposed in \ref{Lipschitz assumptions}, as can be seen in the proof of Lemma 1 of \citep{amorino2020invariant}.
Now we note that assumption $(\textbf{H}^\kappa)$ of \citep{CHEN20176576} follows by \ref{Kappa Assumptions}.
\end{proof}

\begin{proof}[Proof of Corollary \ref{coroll: heat jump}]
Since \ref{Lipschitz assumptions} and \ref{Ergodicity Assumptions} imply that $b^*$ is bounded, arguing as in the proof of Lemma 1 of \citep{amorino2020invariant} and using Lemma \ref{Chen and Applebaum assumptions} entails that $b^\ast$ belongs to the Kato class $\mathbb{K}_2$ for $d \geq 2$. For the definition of $\mathbb{K}_2$, see (2.28) in \citep{CHEN20176576}. 
Existence of transition densities and the heat kernel estimate now follow directly from Corollary 1.5 of \citep{CHEN20176576} and Lemma \ref{Chen and Applebaum assumptions} for $d \geq 2$ and as described in Lemma 1 of \citep{amorino2020invariant}, the same conclusions may be drawn for dimension $d=1$ by adapting the arguments in \cite{CHEN20176576}. Now note that \eqref{transition density bounds}, $t\leq 1$ and $\alpha\in(0,2)$ imply
\begin{align*}
p_t(x,y)&\leq C(t^{-d/2}\exp(-\Vert x-y\Vert^2/(\lambda t))+\Vert \kappa_\alpha\Vert_\infty t(\Vert x-y\Vert +t^{1/2})^{-d-\alpha})\\
&\leq C(t^{-d/2}+t^{1-(d+\alpha)/2})\leq Ct^{-d/2},
\end{align*}
where the value of $C$ changes from line to line. This completes the proof. 
\end{proof}

\begin{proof}[Proof of Proposition \ref{Exponential Ergodicity}]
To verify the assertion, we show that the solution of \eqref{SDE} $\X$ satisfies the assumptions of Theorem 2.2 (ii) of \citep{MASUDA200735} which are Assumption 1, 2(a)' and 3* of \citep{MASUDA200735} and \citep{MASUDAerratum}, respectively. 
Assumption 1 follows directly from \ref{Lipschitz assumptions}. 
Now, define $b^\ast_u(x)\coloneqq b^\ast(x)-\int_{u<\Vert z\Vert\leq1} \gamma(x)z\,\nu(\mathrm{d}z)=b(x)-\int_{\Vert z\Vert >u}\gamma(x)z\,\nu(\mathrm{d}z)$, and let the diffusion process $Y^u=(Y^u_t)_{t\geq0}$ be given by \[Y^u_t=x+\int_0^tb^\ast_u(x)(Y^u_s)\d s+\int_0^t\sigma(Y^u_s)\d W_s.\]
For Assumption 2(a)' to be fulfilled, we first have to show that, for any $u\in(0,1)$, there exists $\Delta>0$ such that $\Pro_x(Y^u_\Delta\in B)>0$ for any $x\in\R^d$ and any nonempty open set $B\subset \R^d$.  Since $Y^u$ is a continuous diffusion process with bounded and Lipschitz coeffcients $b^\ast_u, \sigma$ and $a = \sigma \sigma^T$ is uniformly elliptic, it follows from classical results, see e.g. \cite[Theorem A]{sheu1991}, that for any $x \in \R^d$ and $t > 0$, the transition function $P^u_t(x,\cdot)$ of $Y^u$ has a transition density with full support and hence any $\Delta$-skeleton of $Y^u$ is open set irreducible, showing that Assumtpion 2(a)' is in place.
It remains to show that Assumption 3* of \citep{MASUDA200735} is satisfied which is, that there exists a function $V\in Q^*,$ where
\begin{align*}
Q^*\coloneqq \Big\{&f\colon \R^d\to\R_+:  f\in \mathcal{C}^2,f(x)\to\infty \textrm{ as } \Vert x\Vert \to \infty,\textrm{ and there exists a locally bounded,}\\&\textrm{ measurable function $\bar{f}$, }\textrm{such that } \int_{\Vert z\Vert >1} f(x+\gamma(x)z) \,\nu(\mathrm{d} z)\leq \bar{f}(x), \, \forall x\in\R^d   \Big\},
\end{align*}
such that there are constants $c_1,c_2>0,$ for which the Lyapunov drift criterion
\begin{equation}\label{Lyapunov}
\mathsf{A}V\leq -c_1V+c_2
\end{equation}
holds, where $\mathsf{A}$ denotes the extended generator of $\X$ acting on $Q^*$ by
\begin{align*}
\mathsf{A}f(x)&=\qv{\nabla f(x),b^\ast(x)}+\tfrac{1}{2}\operatorname{tr}(\nabla^2f(x)\sigma(x)\sigma^T(x))\\
&\hspace*{3em}+\int_{\R^d}f(x+\gamma(x)z)-f(x)-\mathbf{1}_{\Vert z\Vert \leq 1} \qv{\nabla f(x),\gamma(x)z}\nu(\mathrm{d}z), \quad x \in \R^d, f \in Q^\ast.
\end{align*} 
Now, for $\eta\in(0,\eta_0c_\gamma^{-1}\land 1)$, where $c_\gamma \coloneqq \lVert \gamma \rVert_\infty$, let $V^\eta$ be a positive and increasing function in $\mathcal{C}^2(\R^d,\R)$ such that $V^\eta=\e^{\eta\Vert x\Vert}$ for all $\Vert x\Vert>c_V$, where $c_V>0$. Then, it holds
for $i\neq j \in\{1,\ldots,d\}$ and $\Vert x\Vert >c_V,$
\begin{equation}
\begin{split}\label{Derivatives}
    \partial_i V^\eta(x)&= \eta\e^{\eta\Vert x\Vert}\frac{x_i}{\Vert x\Vert}, \\
    \partial^2_{ij}V^\eta(x)&=\eta^2\e^{\eta\Vert x\Vert}\frac{x_ix_j}{\Vert x\Vert^2}-\eta\e^{\eta\Vert x\Vert}\frac{x_ix_j}{\Vert x\Vert^3} + \eta\e^{\eta\Vert x\Vert}\Vert x\Vert^{-1}\delta_{ij},
\end{split}
\end{equation}
where $\delta_{ij}$ denotes the Kronecker delta. Furthermore, since $V^{\eta}\in \mathcal{C}^2(\R^d;\R)$, for $i,j\in\{1,\ldots,d\}$ the functions $V^{\eta},\partial_iV^{\eta},\partial^2_{ij}V^{\eta}$ are bounded by a constant $c_D>0$ for $\Vert x\Vert \leq c_V$ and
hence
\begin{align*}
\int_{\Vert z\Vert >1} V^\eta(x+\gamma(x)z) \, \nu(\mathrm{d} z)&\leq \int_{\Vert z\Vert >1} \big(\e^{\eta\Vert x+\gamma(x)z\Vert}+c_D \big) \,\nu(\mathrm{d} z)
\\
&\leq \e^{\eta\Vert x\Vert} \int_{\Vert z\Vert >1} \e^{c_\gamma\eta\Vert z\Vert}\nu(\mathrm{d} z)+c_D \nu(\R^d \setminus B_1),
\end{align*}
implying that $V^\eta_\cS\in Q^*$ for all $\eta\leq \tfrac{\eta_0}{c_\gamma}$.
This last condition is satisfied by our choice of $\eta$. To conclude the proof, the only thing left to show is that there exists $0<\eta\leq\tfrac{\eta_0}{c_\gamma}$ such that \eqref{Lyapunov} holds for $V^\eta$. 
Note that, by the mean value theorem, the definition of $b^\ast$ and the Cauchy--Schwarz inequality, we have for any $f \in Q^\ast$
\begin{align*}
\mathsf{A}f(x)
&=\qv{\nabla f(x),b(x)}+\tfrac{1}{2}\operatorname{tr}(\nabla^2f(x)\sigma(x)\sigma^\top (x))+\int_{\R^d}f(x+\gamma(x)z)-f(x)- \qv{\nabla f(x),\gamma(x)z}\nu(\mathrm{d}z)
\\
&\leq \qv{\nabla f(x),b(x)}+\tfrac{1}{2}\operatorname{tr}(\nabla^2f(x)\sigma(x)\sigma^\top (x))\\
&\qquad+\int_{\R^d}\sup_{t\in[0,1]}\Vert\nabla f(x+t\gamma(x)z)-\nabla f(x)\Vert\Vert\gamma(x)z\Vert\nu(\mathrm{d}z)
\\
&\leq \mathsf{A}_cf(x)+\mathsf{A}_df(x),
\end{align*}
where, for $\mathsf{H}^2f(x)$ denoting the Hessian of $f$ evaluated at $x$,
\begin{align*}
\mathsf{A}_cf(x)&\coloneqq\qv{\nabla f(x),b(x)}+\tfrac{1}{2}\operatorname{tr}(\nabla^2f(x)\sigma(x)\sigma^T(x)),\\
\mathsf{A}_df(x)&\coloneqq c_\gamma^2\int_{\R^d}\sup_{t\in[0,1]} \Vert\mathsf{H}^2f(x+t\gamma(x)z)\Vert\Vert z\Vert^2\, \nu(\mathrm{d}z).
\end{align*}
We start by investigating the jump part.
By \eqref{Derivatives} and the fact that the operator norm can be bounded by the Frobenius norm $\Vert \cdot\Vert_F$, we get for $\Vert x\Vert>c_V$
\begin{align*}
\Vert \mathsf{H}^2V^\eta(x)\Vert&\leq \Vert \mathsf{H}^2\e^{\eta\Vert x\Vert}\Vert_F=\left(\sum_{i,j=1}^{d}\left(\eta^2\e^{\eta\Vert x\Vert}\frac{x_ix_j}{\Vert x\Vert^2}-\eta\e^{\eta\Vert x\Vert}\frac{x_ix_j}{\Vert x\Vert^3}+\eta\e^{\eta\Vert x\Vert}\Vert x\Vert^{-1}\delta_{ij}\right)^2 \right)^{\frac{1}{2}}
\\
&\leq 2\eta\e^{\eta\Vert x\Vert}\left(\sum_{i,j=1}^{d}\left(\eta^2\frac{x_i^2x_j^2}{\Vert x\Vert^4}+\frac{x_i^2x_j^2}{\Vert x\Vert^6}+\Vert x\Vert^{-2}\delta_{ij}\right) \right)^{\frac{1}{2}}
\leq 2^{3/2}\sqrt d\eta\e^{\eta\Vert x\Vert}\left(\eta^2+2\Vert x\Vert^{-2}\right)^{\frac{1}{2}}.
\end{align*}
Since we can choose $c_V$ to be large, we can without loss of generality assume $c_V\geq \sqrt{2} \eta^{-1}$ and, additionally, $V^\eta \in \mathcal{C}^2$ implies that there exists a real-valued function $c_\mathsf{H}(\eta)>0$ on $(0,\infty)$ such that $\Vert \mathsf{H}^2V^\eta(x)\Vert< c_\mathsf{H}(\eta)$ for all $\Vert x\Vert \leq c_V$.  Thus, we have 
$\Vert \mathsf{H}^2V^\eta(x)\Vert\leq 4 \sqrt d\eta^2\e^{\eta\Vert x\Vert}+c_\mathsf{H}(\eta)$, $x \in \R^d$,
and we can conclude
\begin{equation}
\label{ineq: Ad}
\begin{split}
\mathsf{A}_dV^\eta(x) 
&\leq 4c_\gamma^2\sqrt d\eta^2\int_{\R^d}\sup_{t\in[0,1]}\e^{\eta\Vert x+t\gamma(x)z\Vert} \Vert z\Vert^2\nu(\mathrm{d}z)+c_\gamma^2c_\mathsf{H}(\eta)\int_{\R^d} \Vert z\Vert^2\nu(\mathrm{d}z)
\\
&\leq \eta^2\e^{\eta\Vert x\Vert}4c_\gamma^2\sqrt d\int_{\R^d}\e^{\eta_0 \Vert z\Vert} \Vert z\Vert^2\nu(\mathrm{d}z)+ c_\gamma^2c_\mathsf{H}(\eta)\int_{\R^d} \Vert z\Vert^2\nu(\mathrm{d}z)\eqqcolon c_{d,1} \eta^2\e^{\eta\Vert x\Vert} +c_{d,2}(\eta),
\end{split}
\end{equation}
where $c_{d,1},c_{d,2}(\eta)$ are positive and finite because of \ref{Ergodicity Assumptions} and $\eta<\eta_0c_\gamma^{-1}$. 
Now we turn our attention to the continuous part. From now on, without loss of generality, we assume that $c_V\geq c_1$ in \ref{Ergodicity Assumptions}. Then, for $\Vert x\Vert>c_V\geq\eta^{-1},$ we have by \ref{Lipschitz assumptions}, \ref{Ergodicity Assumptions} and \eqref{Derivatives}
\begin{align*}
\mathsf{A}_cV^\eta(x)
&\leq-c_1\eta \e^{\eta\Vert x\Vert} +\frac{c_2}{2}\sum_{k=1}^{d}\Big\vert\eta^2\e^{\eta\Vert x\Vert}\frac{x_i^2}{\Vert x\Vert^2}+\eta\e^{\eta\Vert x\Vert}\Vert x\Vert^{-1}-\eta\e^{\eta\Vert x\Vert}\frac{x_i^2}{\Vert x\Vert ^3}\Big\vert
\\
&\leq \eta \e^{\eta\Vert x\Vert}\left(-c_1+ \frac{3c_2d}{2}\eta\right),
\end{align*}
and since $V^\eta\in \mathcal{C}^2(\R^d;\R)$, there exists a real-valued function $c_c(\eta)$ on $(0,\infty)$ such that $\mathsf{A}_cV^\eta(x)\leq c_c(\eta)$ for all $\Vert x\Vert \leq c_V$. Hence, we have
\begin{equation}\label{ineq: Ac}
\mathsf{A}_cV^\eta(x)\leq \eta \e^{\eta\Vert x\Vert}\left(-c_1+ \frac{3c_2d}{2}\eta\right)+c_C(\eta)+c_1\e^{c_V} \eqqcolon \eta \e^{\eta\Vert x\Vert}\left(-c_1+ \frac{3c_2d}{2}\eta\right) + c_{c,1}(\eta),
\end{equation}
where we used that $\eta<1$, by assumption. Combining \eqref{ineq: Ad} and \eqref{ineq: Ac} yields
\begin{align*}
\mathsf{A}V^\eta(x)
&\leq \eta \e^{\eta\Vert x\Vert}\left(-c_1+ \eta\left(\frac{3c_2d}{2} + c_{d,1}\right)\right)+ c_{d,2}(\eta)+c_{c,1}(\eta).
\end{align*}
Choosing 
$\eta^\ast=1\land \eta_0c_\gamma^{-1}\land \frac{c_1}{3c_2d+2c_{d,1}}$ 
implies 
\[\mathsf{A}V^{\eta^*}(x)\leq \frac{c_1\eta^*}{2} \e^{\eta\Vert x\Vert}+ c_{d,2}(\eta^*)+c_{c,1}(\eta^*),\] 
and thus \eqref{Lyapunov} holds for $V^{\eta^*} \in Q^\ast.$ 
Now, Theorem 2.2 (ii) and Proposition 3.8 of \citep{MASUDA200735} show the required assertion.
\end{proof}

\printbibliography
\end{document}